\documentclass[pdflatex,sn-mathphys-num]{sn-jnl}
\usepackage{amsmath,amssymb,amsthm,mathtools}
\usepackage{graphicx}
\usepackage{xcolor}
\usepackage{placeins}
\usepackage{microtype}
\usepackage{xurl}
\providecommand{\burl}[1]{\url{#1}}
\hypersetup{hidelinks}
\numberwithin{equation}{section}
\theoremstyle{thmstyleone}
\newtheorem{theorem}{Theorem}[section]
\newtheorem{lemma}[theorem]{Lemma}
\newtheorem{proposition}[theorem]{Proposition}
\newtheorem{corollary}[theorem]{Corollary}
\theoremstyle{thmstyletwo}
\newtheorem{remark}{Remark}[section]
\allowdisplaybreaks
\begin{document}

\title{Uniform-in-time approximation and convergence of invariant measures
for the damped stochastic Korteweg--de Vries equation}
\author[1]{\fnm{Junjie} \sur{Li}}\email{652023210010@smail.nju.edu.cn}
\author*[1]{\fnm{Chun} \sur{Li}}\email{lichun@nju.edu.cn}

\author[2]{\fnm{Tau} \sur{Zhou}} \email{zt@hunnu.edu.cn}
\author[1]{\fnm{Zhao} \sur{Wang}}\email{zhaowang@smail.nju.edu.cn}
\affil[1]{\orgdiv{School of Mathematics}, \orgname{Nanjing University},
\orgaddress{\city{Nanjing}, \postcode{210093}, \country{P. R. China}}}

\affil[2]{\orgdiv{Key Laboratory of Computing and Stochastic Mathematics (Ministry of Education), School of Mathematics and Statistics}, \orgname{Hunan Normal University},
	\orgaddress{\city{Changsha}, \postcode{410081}, \country{P. R. China}}}

\abstract{To quantitatively characterize the long-time dynamics of the periodic damped
stochastic Korteweg--de Vries (sKdV) equation driven by additive noise, we
investigate the uniform-in-time error estimates for a Lie--Trotter
operator splitting approximation. This splitting combines the exact
deterministic KdV flow with the exact Ornstein--Uhlenbeck solution map for
linear damping and additive forcing. Long-time error analysis for the sKdV
equation is highly nontrivial because of the interplay among the lack of
dissipative smoothing, the loss of one spatial derivative arising from the
Burgers-type nonlinearity, and the fluctuations of stochastic forcing. To
overcome these difficulties, we develop new strategies based on exponential
Lyapunov estimates, the continuous dependence estimate, and the local error
decomposition. We establish strong and weak order \(1\)
convergence uniformly in time under sufficiently large damping, as well as
strong order \(1\) convergence on finite time intervals under a weaker
damping condition. Under stronger assumptions, the splitting approximation
admits a unique invariant probability measure on \(H_0^2\), which approximates
the invariant measure of the exact dynamics with order \(1\). To the best of our knowledge, these constitute the first
uniform-in-time strong error and quantitative invariant-measure
approximation results for the periodic damped sKdV equation with additive
noise. Our results provide an avenue for long-time simulation of the damped sKdV equation.}

\maketitle

\section{Introduction}
\label{sec:introduction}

Let \((\Omega,\mathcal F,(\mathcal F_t)_{t\geq0},\mathbb P)\) be a filtered
probability space satisfying the usual conditions. On this probability space
we consider the periodic damped stochastic Korteweg--de Vries (sKdV) equation
\begin{equation}
 \mathrm du+\left(\partial_x^3u+\frac12\partial_x(u^2)+\lambda u\right)\mathrm dt
 =\Phi\,\mathrm dW,
 \qquad
 x\in\mathbb T=\mathbb R/(2\pi\mathbb Z),
 \qquad \lambda>0.
 \label{eq:intro-skdv}
\end{equation}
Here \(W\) is a cylindrical Wiener process on a separable Hilbert space
\(\mathfrak U\), and \(\Phi\) is a deterministic Hilbert--Schmidt operator
with values in a periodic Sobolev space.
Understanding long-time dynamics is a central problem in the study of
stochastic evolution equations. For the damped sKdV equation, several important qualitative
properties of its long-time behavior have been established~\cite{EkrenKukavicaZiane2018,GHMR2024}, while obtaining
a reliable quantitative characterization remains challenging. Such a
quantitative study requires approximations whose errors can be
controlled over arbitrarily long times. We therefore study uniform-in-time
approximation of \eqref{eq:intro-skdv} and quantitative approximation of its
invariant measure. To construct such an approximation, we employ operator
splitting, which approximates an evolution problem by composing the solution
maps of simpler subproblems; see, e.g., \cite{Holden2010}. Accordingly, we
separate the deterministic KdV dynamics from the linear damping and
stochastic forcing.

Letting \(\mathcal K_h\) denote the exact periodic KdV flow over one time
step, we consider the KdV--Ornstein--Uhlenbeck (KdV--OU) Lie--Trotter
approximation
\begin{equation}
 U^0=u_0,\qquad
 U^{n+1}
 =
 e^{-\lambda h}\mathcal K_hU^n
 +
 \int_{t_n}^{t_{n+1}}
 e^{-\lambda(t_{n+1}-r)}\Phi\,\mathrm dW(r),
 \qquad t_n=nh.
 \label{eq:intro-scheme}
\end{equation}
Each time step composes the exact deterministic KdV flow with the exact
Ornstein--Uhlenbeck solution map for the linear damping and the additive
forcing. By solving both steps exactly, we isolate the splitting error in
time and study its effect on the long-time approximation of solutions and
invariant measures. The resulting error estimates provide an avenue for
long-time simulation of the damped sKdV equation when suitable numerical
realizations are used for the splitting substeps. For deterministic KdV,
the classical Airy--Burgers splitting has been shown to converge on finite
time intervals for sufficiently regular data
\cite{HoldenLubichRisebro2013,HKRT2011}. In the stochastic setting considered
here, however, this decomposition encounters a structural obstruction:
under a nondegeneracy condition on the noise, we show that the classical
Airy--Burgers splitting almost surely breaks down after finitely many steps.
This obstruction is one reason for retaining the globally defined KdV flow
in \eqref{eq:intro-scheme}, rather than separating the Airy part from the
nonlinearity.

In the error analysis of \eqref{eq:intro-scheme}, the main difficulty is to
control the propagation of local splitting errors uniformly in time. The
Airy group preserves Sobolev norms and provides no dissipative smoothing,
while the nonlinear transport term carries one spatial derivative.
Consequently, the continuous dependence estimate for the KdV flow involves a
growth factor depending on higher Sobolev norms of the solution. To bound
this growth over arbitrarily many steps, we use the coercive modifications
of the KdV invariants introduced in~\cite{GHMR2024}. The KdV step preserves
these modified invariants exactly, while the damping and additive forcing
enter only through the Ornstein--Uhlenbeck step. Accordingly, the modified
invariants evolve only during the Ornstein--Uhlenbeck step, and the
splitting approximation satisfies the same Lyapunov estimates as the exact
dynamics. Polynomial moment bounds give the Sobolev norm estimates needed
in the local error analysis, while exponential Lyapunov bounds, together
with the continuous dependence estimate and the linear damping, control the
accumulation of these growth factors over long times. Each local splitting
error is further decomposed into its conditional mean, whose
\(L^2(\mathbb T)\)-norm is \(O(h^2)\), and a zero-mean part of
root-mean-square size \(O(h^{3/2})\). These ingredients yield the
exponentially weighted mean-square estimate stated below.

Assume that \(u_0\in H^6_0(\mathbb T)\) is deterministic and that \(\Phi\in
L_2(\mathfrak U,H^6_0(\mathbb T))\), where the subscript \(0\) denotes the
mean-zero subspace. We drive the exact solution and the splitting
approximation by the same Wiener process. Let \(G\) be the Lyapunov function
introduced in Section~\ref{sec:discrete-lyapunov}, with \(a>0\) chosen as in
Section~\ref{sec:strong-convergence}. For sufficiently large fixed damping
and sufficiently small \(h\), Theorem~\ref{thm:uniform-strong} gives
\[
 \sup_{n\ge0}
 \mathbb E\!\left[
 e^{2aG(U^n)}
 \|u(nh)-U^n\|_{L^2}^2
 \right]
 \le Ch^2.
\]
The constant \(C\) is independent of the step index, the final time, and the
step size. Since the weight is at least one, the splitting approximation has strong order
\(1\) uniformly in time, and the same weighted estimate also yields weak
order \(1\) uniformly in time. Under stronger conditions, the splitting
approximation admits a unique invariant probability measure on \(H^2_0\),
which approximates the invariant measure of the exact dynamics with error
\(O(h)\), both in a weighted dual distance and in the 2-Wasserstein distance
induced by the \(L^2\) norm. To our knowledge, these are the first
uniform-in-time strong error and quantitative invariant-measure
approximation results for this stochastic model. On finite time intervals, strong order \(1\) holds under a weaker
damping condition.

The present work is situated in the broader context of theoretical research
on stochastic KdV equations. In the periodic setting, de Bouard, Debussche,
and Tsutsumi proved local well-posedness with additive noise that is white
in time and almost white in space \cite{deBouardDebusscheTsutsumi2005}. Oh subsequently
established local well-posedness with additive space--time white noise in a
Besov-type space \(\widehat b^s_{p,\infty}(\mathbb T)\) \cite{Oh2009}.
On the real line, Greco, Oh, and Tsugawa obtained global well-posedness in
\(L^2(\mathbb R)\), removing an additional homogeneous Sobolev regularity
assumption in the earlier work of de Bouard, Debussche, and Tsutsumi
\cite{deBouardDebusscheTsutsumi1999,GrecoOhTsugawa2026}. For the damped sKdV
equation, Ekren et al.\ established the existence of an invariant measure on
the real line \cite{EkrenKukavicaZiane2018}. In the periodic setting,
Glatt-Holtz, Martinez, and Richards \cite{GHMR2024} made novel use of the
KdV conservation laws to develop a long-time statistical theory for the
damped sKdV equation. They established a Lyapunov structure for higher-order
Sobolev norms, which yields the existence of invariant measures, and proved
uniqueness when sufficiently many low Fourier modes are stochastically
forced. For sufficiently large damping, they further established a spectral
gap with respect to a Wasserstein-like distance and consequently obtained
uniqueness and exponential mixing. Their results play an important role in
the present long-time approximation analysis and motivate the choice of
\(G\) in our exponential weight.


On the numerical side, sKdV has been studied under several discretizations
and in different regimes. Debussche and Printems used a least-squares
finite-element method to study noise-perturbed solitons
\cite{DebusschePrintems1999} and later proved convergence in probability for
an implicit Crank--Nicolson time discretization on the real line, without a
rate in the step size \cite{DebusschePrintems2006}. D'Ambrosio and
Di Giovacchino studied numerical conservation properties of stochastic theta
methods for a noise that does not depend on the space variable
\cite{DAmbrosioDiGiovacchino2023}. More recently, Cui et al.\ obtained
explicit strong error bounds on finite time intervals for an approximation
of the periodic sKdV equation with small additive noise
\cite{CuiDAmbrosioDiGiovacchinoSun2026}. The numerical
ergodic behavior of the damped sKdV equation was studied in
\cite{LiLiExponential} using an explicit exponential integrator with
artificial viscosity. The Galilean splittings developed in our earlier work
\cite{LiWangLiGalilean} separate the full deterministic KdV dynamics from
the stochastic forcing. The present work builds on this splitting idea.
These studies focus on finite-time convergence, conservation properties, or
numerical ergodic behavior, rather than uniform-in-time error estimates and
quantitative invariant-measure approximation.

Related long-time numerical results have been established for several other
classes of stochastic evolution equations. For damped stochastic nonlinear
Schr\"odinger equations, uniform-in-time strong and weak convergence has been
obtained for a splitting method \cite{CuiHong2018}, while approximation of
invariant measures has been studied for an additive-noise model
\cite{ChenHongWang2017}. Long-time error bounds on parameter-dependent
intervals are also available in small-parameter regimes
\cite{DiGiovacchino2026}. In dissipative settings, related long-time analyses
have been developed for monotone and superlinear SPDEs, including stochastic
Allen--Cahn and Cahn--Hilliard equations
\cite{CaiLiu2026,HeDengZhangCao2026,LiuNumericalErgodicity2026}. For parabolic
SPDEs with non-globally Lipschitz coefficients, uniform-in-time weak error
estimates have been established for explicit full discretizations
\cite{JiangWang2025}, and strong and weak estimates for linearly implicit
finite element schemes \cite{QiWang2026}. Approximation of invariant
measures has also been studied through Kolmogorov or Poisson equations
\cite{Brehier2014,BrehierKopec2017} and through Lyapunov bounds combined with
quantitative contraction, including applications to stochastic
Navier--Stokes equations \cite{GlattHoltzMondaini2025}. These analyses
typically rely on smoothing properties of the linear part, monotonicity
properties of the drift, or nonlinearities that carry no spatial derivative,
whereas the damped sKdV equation has a quite different stability structure.

The remainder of the paper is organized as follows. Section~2 defines the
KdV--OU splitting, establishes its Feller property and exact conditional
\(L^2\)-moment identity, and studies the classical Airy--Burgers obstruction.
Section~3 introduces the modified invariants and establishes the Lyapunov
estimates used in the convergence analysis. Section~4 proves the local error
bounds and the strong convergence results on finite time intervals and
uniformly in time. Section~5 establishes the weak error bounds and the
approximation of invariant measures. Section~6 presents the numerical
illustrations, and Section~7 concludes the paper.

\section{Splittings for stochastic KdV}
\label{sec:stochastic-setting}

In this section, we define the KdV--OU splitting, in which both steps are
solved exactly, and derive the basic properties of the resulting
approximation: it is well defined and adapted on \(H^m\), it is a
time-homogeneous Markov process with Feller transition kernel, it reproduces
the spatial mean of the exact solution, and it satisfies an exact conditional
\(L^2\)-moment identity.  For the classical Airy--Burgers splitting, whose
Burgers step is defined only on a proper subset of the state space, we then
prove that, under a nondegeneracy condition on the noise, the composition
almost surely breaks down after finitely many steps.

\subsection{KdV--OU Lie--Trotter splitting}

Let \(\mathbb T=\mathbb R/(2\pi\mathbb Z)\), and let
\(H^m=H^m(\mathbb T;\mathbb R)\) denote the periodic Sobolev space
of order \(m\).  Define the spatial mean by
\[
 \Pi_0v=\frac{1}{2\pi}\int_{\mathbb T}v(x)\,\mathrm dx,
 \qquad
 H^m_0=\{v\in H^m:\Pi_0v=0\}.
\]
Let \(\mathfrak U\) be a separable Hilbert space.  We denote by
\(L_2(\mathfrak U,H^m)\) the space of Hilbert--Schmidt operators from
\(\mathfrak U\) to \(H^m\), equipped with the Hilbert--Schmidt norm
\(\|\cdot\|_{L_2(\mathfrak U,H^m)}\).  Let
\((\Omega,\mathcal F,(\mathcal F_t)_{t\geq0},\mathbb P)\) be a filtered
probability space satisfying the usual conditions, and let \(W\) be a
cylindrical \(\mathfrak U\)-Wiener process with respect to
\((\mathcal F_t)_{t\geq0}\).  The noise operator \(\Phi\) is deterministic,
and all stochastic integrals are understood in the It\^o sense.  We study
\begin{equation}
 \mathrm du=\{F(u)-\lambda u\}\,\mathrm dt+\Phi\,\mathrm dW,
 \qquad
 F(v)=-\partial_x^3v-\frac12\partial_x(v^2),
 \qquad \lambda>0.
 \label{eq:skdv}
\end{equation}

Fix an integer \(m\geq2\) and assume \(\Phi\in L_2(\mathfrak U,H^m)\).  Let
\((\mathcal K_t)_{t\in\mathbb R}\) denote the global periodic KdV flow on
\(H^m\) \cite{KappelerTopalov2006}.  For \(u_0\in H^m_0\) and
\(\Phi\in L_2(\mathfrak U,H^m_0)\), global pathwise well-posedness of
\eqref{eq:skdv} and continuous dependence on the initial data follow from
\cite[Proposition~2.6]{GHMR2024}.  For general spatial means, separate the
scalar Ornstein--Uhlenbeck process \(\Pi_0u\) from \(u-\Pi_0u\).  The latter
satisfies the mean-zero equation with noise coefficient \(\Phi-\Pi_0\Phi\)
and the additional transport term \(-(\Pi_0u)\partial_x(u-\Pi_0u)\).
Since \(\Pi_0u\) is continuous and this transport term contributes zero to
the evolution of every translation-invariant KdV functional, the
regularization and energy arguments extend to this equation, yielding
global pathwise well-posedness and continuous dependence on the initial
data in \(H^m\).

Throughout, \(C\) denotes a positive
constant whose value may change from one occurrence to the next; the
dependence of the constants in the main estimates is stated in each result.

For \(0\leq s<t\), define the exact Ornstein--Uhlenbeck solution map by
\begin{equation}
 \mathcal O^\omega_{t,s}(v)
 =e^{-\lambda(t-s)}v
  +\int_s^t e^{-\lambda(t-r)}\Phi\,\mathrm dW(r).
 \label{eq:ou-map}
\end{equation}
Let \(\mathcal S^\omega_{t,s}(x)\) denote the solution of \eqref{eq:skdv} at
time \(t\), started from \(x\) at time \(s\) and driven by \(W\) on \([s,t]\).
For \(h>0\), define the KdV--OU Lie--Trotter one-step map by
\begin{equation}
 \mathcal L^\omega_{s+h,s}(x)
 =\mathcal O^\omega_{s+h,s}(\mathcal K_hx)
 =e^{-\lambda h}\mathcal K_hx
  +\int_s^{s+h}e^{-\lambda(s+h-r)}\Phi\,\mathrm dW(r).
 \label{eq:lie-map}
\end{equation}
Thus each step first applies the KdV flow over a time \(h\), and then the
exact Ornstein--Uhlenbeck step.  With \(t_n=nh\), the splitting approximation
\((U^n)_{n\geq0}\) is defined by
\begin{equation}
 U^{n+1}=\mathcal L^\omega_{t_{n+1},t_n}(U^n).
 \label{eq:splitting-approximation}
\end{equation}
Here \(U^0\) denotes the initial state.

For a bounded measurable function \(\varphi:H^m\to\mathbb R\), define
\[
 P_t\varphi(x)=\mathbb E[\varphi(\mathcal S^\omega_{t,0}x)],
 \qquad
 Q_h\varphi(x)=\mathbb E[\varphi(\mathcal L^\omega_{h,0}x)].
\]
For a probability measure \(\nu\) on \(H^m\), define \(\nu P_t\) and
\(\nu Q_h\) by
\[
 (\nu P_t)(\varphi)=\nu(P_t\varphi),
 \qquad
 (\nu Q_h)(\varphi)=\nu(Q_h\varphi).
\]

The following proposition collects the basic properties of the splitting.

\begin{proposition}[Basic properties of the KdV--OU splitting]
\label{prop:basic-structure}
Under the stated well-posedness assumptions, for every \(h>0\) and every
\(H^m\)-valued \(\mathcal F_0\)-measurable initial state \(U^0\), the
splitting approximation \((U^n)_{n\geq0}\) defined by
\eqref{eq:splitting-approximation} is well defined for all \(n\geq0\) and
adapted.  Moreover, it is a time-homogeneous discrete-time Markov process on
\(H^m\) with Feller transition kernel \(Q_h\).

If the exact solution and the splitting approximation start from the same
initial state and are driven by the same Wiener process, then
\begin{equation}
 \Pi_0U^{n+1}=e^{-\lambda h}\Pi_0U^n
 +\int_{t_n}^{t_{n+1}}e^{-\lambda(t_{n+1}-r)}\Pi_0\Phi\,\mathrm dW(r)
 =\Pi_0u(t_{n+1})
 \label{eq:mean-exact}
\end{equation}
for every \(n\geq0\), almost surely.  The kernel \(Q_h\) leaves \(H^m_0\)
invariant if and only if \(\Pi_0\Phi=0\).  Consequently, if
\(\Pi_0\Phi=0\) and \(U^0\in H^m_0\) almost surely, then
\(U^n\in H^m_0\) almost surely for every \(n\geq0\).
\end{proposition}

\begin{proof}
The periodic KdV flow is globally defined and continuous on \(H^m\), while
the stochastic integral in \eqref{eq:lie-map} is a Gaussian random variable
with values in \(H^m\) and mean zero.  Hence the one-step map is measurable and
defined on all of \(H^m\).  By induction using
\eqref{eq:splitting-approximation}, \(U^n\) is
\(\mathcal F_{t_n}\)-measurable.  The stochastic convolution over
\([t_n,t_{n+1}]\) is independent of \(\mathcal F_{t_n}\) and has the same
law for every \(n\).  Therefore, for every bounded measurable \(\varphi\),
\[
 \mathbb E[\varphi(U^{n+1})\mid\mathcal F_{t_n}]
 =Q_h\varphi(U^n),
\]
which proves the time-homogeneous Markov property.  If \(v_k\to v\) in
\(H^m\), then \(\mathcal K_hv_k\to\mathcal K_hv\).  Using the same stochastic
convolution in \eqref{eq:lie-map}, continuity of \(\varphi\) and dominated
convergence yield \(Q_h\varphi(v_k)\to Q_h\varphi(v)\) for every
\(\varphi\in C_b(H^m)\).  Hence \(Q_h\) is Feller.

Since \(\Pi_0F(v)=0\) for periodic \(v\), the KdV flow preserves the spatial
mean.  Applying \(\Pi_0\) to \eqref{eq:skdv} and \eqref{eq:lie-map} yields the
same scalar Ornstein--Uhlenbeck update at the times \(t_n\), which proves
\eqref{eq:mean-exact}.  For a mean-zero input, the output mean is a centered
Gaussian random variable with variance
\[
 \frac{1-e^{-2\lambda h}}{2\lambda}
 \|\Pi_0\Phi\|_{L_2(\mathfrak U,\mathbb R)}^2.
\]
It vanishes almost surely if and only if \(\Pi_0\Phi=0\), which proves the
invariance criterion.  Induction then gives \(U^n\in H^m_0\) almost surely
for every \(n\geq0\) whenever \(U^0\in H^m_0\) almost surely.
\end{proof}

Since the KdV step conserves the \(L^2\)-norm, the splitting approximation
obeys the following exact conditional moment identity.

\begin{proposition}[Exact conditional \(L^2\)-moment identity]
\label{prop:l2-balance}
For every \(n\geq0\),
\begin{equation}
 \mathbb E[\|U^{n+1}\|_{L^2}^2\mid\mathcal F_{t_n}]
 =e^{-2\lambda h}\|U^n\|_{L^2}^2
 +\frac{1-e^{-2\lambda h}}{2\lambda}
   \|\Phi\|_{L_2(\mathfrak U,L^2)}^2.
 \label{eq:split-l2-balance}
\end{equation}
Consequently,
\begin{equation}
 \mathbb E[\|U^n\|_{L^2}^2]
 =e^{-2\lambda nh}\mathbb E[\|U^0\|_{L^2}^2]
 +\frac{1-e^{-2\lambda nh}}{2\lambda}
   \|\Phi\|_{L_2(\mathfrak U,L^2)}^2.
 \label{eq:split-l2-iteration}
\end{equation}
\end{proposition}

\begin{proof}
Conditioning on \(\mathcal F_{t_n}\) in \eqref{eq:lie-map}, the stochastic
convolution has mean zero and is independent of \(\mathcal F_{t_n}\), so the
cross term vanishes.  Conservation of the \(L^2\)-norm under the KdV flow gives
\[
 \|\mathcal K_hU^n\|_{L^2}=\|U^n\|_{L^2},
\]
and It\^o's isometry yields the last term in
\eqref{eq:split-l2-balance}.  Taking expectations in
\eqref{eq:split-l2-balance} and iterating the resulting recurrence yields
\eqref{eq:split-l2-iteration}.
\end{proof}

These identities hold for every \(m\geq2\).  The long-time estimates of
Sections~4 and~5 are carried out on the mean-zero subspace \(H^m_0\), where
the required exponential Lyapunov bounds and contraction estimates for the
exact Markov semigroup are available.

\subsection{Obstruction to the classical Airy--Burgers splitting}
\label{sec:burgers-obstruction}

We now turn to the classical Airy--Burgers splitting, which separates the
Airy part from the nonlinearity.  Its Burgers step, unlike the KdV step, is
defined only on a proper subset of the state space.  The obstruction
established here motivates the use of the globally defined KdV flow in
\eqref{eq:lie-map}; the analysis of Sections~3--5 does not depend on it.

Throughout this subsection, the step size \(h>0\) is fixed.  We consider the
Lie--Trotter composition in which each time step consists of a stochastic
Airy--damping step of length \(h\) followed by a Burgers step of length
\(h\).

Consider the periodic Burgers equation
\begin{equation}
 \partial_tv+v\partial_xv=0,
 \qquad v(0)=f\in H^\sigma,
 \qquad \sigma>\frac32.
 \label{eq:burgers}
\end{equation}
Its characteristic map is \(X_t(\xi)=\xi+tf(\xi)\).  As long as this map is
a diffeomorphism,
\begin{equation}
 v(t,X_t(\xi))=f(\xi),
 \qquad
 (\partial_xv)(t,X_t(\xi))=\frac{f'(\xi)}{1+tf'(\xi)}.
 \label{eq:burgers-characteristics}
\end{equation}
Consequently, the initial data whose classical \(H^\sigma\)-solution exists
on \([0,\tau_{\mathrm B}]\), with \(\tau_{\mathrm B}>0\), are exactly those
in
\begin{equation}
 D_{\tau_{\mathrm B}}=\left\{f\in H^\sigma:
 \inf_{x\in\mathbb T}f'(x)>-\tau_{\mathrm B}^{-1}\right\}.
 \label{eq:burgers-domain}
\end{equation}
Indeed, \(f'\) is continuous on \(\mathbb T\) and therefore attains its
infimum, so for \(f\notin D_{\tau_{\mathrm B}}\) the denominator in
\eqref{eq:burgers-characteristics} vanishes at some time in
\((0,\tau_{\mathrm B}]\) and \(\partial_xv\) blows up there.  This is also why the
inequality in \eqref{eq:burgers-domain} is strict.  Each Burgers step in the
composition above has length \(h\), so the relevant domain is \(D_h\).

For the stochastic Airy--damping step, set
\[
 Z_h=\int_0^h e^{-(h-r)(\partial_x^3+\lambda)}\Phi\,\mathrm dW(r).
\]
If this step starts at time \(t\) from an
\(\mathcal F_t\)-measurable state \(V\), the initial datum for the following
Burgers step is
\begin{equation}
 e^{-h(\partial_x^3+\lambda)}V
 +\int_t^{t+h}
 e^{-(t+h-r)(\partial_x^3+\lambda)}\Phi\,\mathrm dW(r).
 \label{eq:stochastic-airy-output}
\end{equation}
The stochastic convolution in \eqref{eq:stochastic-airy-output} is independent
of \(\mathcal F_t\) and has the same centered Gaussian law as \(Z_h\).  We use
the following nondegeneracy condition on its spatial derivative:
\begin{equation}
 \inf_{x\in\mathbb T}\mathbb E[|\partial_xZ_h(x)|^2]>0.
 \label{eq:airy-noise-nondegeneracy}
\end{equation}
This condition does not follow from \(\Phi\neq0\): if the noise is supported
only on the constant Fourier mode, then \(\partial_xZ_h=0\) and
\eqref{eq:airy-noise-nondegeneracy} fails.  As an example
satisfying \eqref{eq:airy-noise-nondegeneracy}, let \(\mathfrak U=\mathbb R^2\)
and
\[
 \Phi\,\mathrm dW=\gamma\cos(kx)\,\mathrm d\beta_1
             +\gamma\sin(kx)\,\mathrm d\beta_2,
 \qquad \gamma>0,\quad k\in\mathbb Z\setminus\{0\},
\]
where \(\beta_1\) and \(\beta_2\) are independent Brownian motions.  The Airy
evolution acts as a rotation on the span of \(\cos(kx)\) and \(\sin(kx)\), so
the It\^o isometry gives
\[
 \mathbb E[|\partial_xZ_h(x)|^2]
 =\frac{\gamma^2k^2}{2\lambda}(1-e^{-2\lambda h}),
\]
the same positive value at every \(x\).

Under \eqref{eq:airy-noise-nondegeneracy}, each time step produces an initial
datum outside \(D_h\) with a probability bounded below independently of the
current state and of the number of steps already taken.  Iterating this bound
gives the following result.

\begin{theorem}[Almost-sure breakdown of the classical Airy--Burgers splitting]
\label{thm:burgers-breakdown}
Assume that \(Z_h\) is \(H^\sigma\)-valued and satisfies the nondegeneracy
condition \eqref{eq:airy-noise-nondegeneracy}.  Then, with
probability one, the Lie--Trotter composition breaks down after finitely many
steps: a Burgers step has no classical \(H^\sigma\)-solution on its
prescribed time interval.  In particular, it does not define an
\(H^\sigma\)-valued approximation for arbitrarily many steps.
\end{theorem}

\begin{proof}
Set \(q_h=\inf_{x\in\mathbb T}\mathbb E[|\partial_xZ_h(x)|^2]\), which is
positive by \eqref{eq:airy-noise-nondegeneracy}.
Let \(\Psi\) denote the standard normal cumulative distribution
function, and set
\[
 p_h=\Psi\!\left(-\frac{1}{h\sqrt{q_h}}\right)>0.
\]
Let \(A_N\) be the event that the first \(N\) time steps can be completed,
with \(A_0=\Omega\), and let \(t=Nh\).

On \(A_N\), let \(V\) be the \(\mathcal F_t\)-measurable input to the next
stochastic Airy--damping step.  The initial datum \(f\) for the ensuing
Burgers step is the expression in \eqref{eq:stochastic-airy-output}.  Since
\(H^\sigma\hookrightarrow C^1(\mathbb T)\) continuously and \(\mathbb T\) is
compact, the map \(x\mapsto\partial_x(e^{-h(\partial_x^3+\lambda)}V)(x)\) is
continuous and attains its minimum; a standard measurable-selection argument
provides an \(\mathcal F_t\)-measurable minimizer \(x_*\).  This derivative
has spatial mean zero by periodicity, so its minimum is nonpositive:
\[
 \partial_x(e^{-h(\partial_x^3+\lambda)}V)(x_*)\leq0.
\]
Conditioning on \(\mathcal F_t\) fixes \(x_*\), while the stochastic
convolution in \eqref{eq:stochastic-airy-output} is independent of
\(\mathcal F_t\).  Its spatial derivative at \(x_*\) is therefore
conditionally Gaussian with mean zero and variance at least \(q_h\).
If this derivative is at most \(-h^{-1}\), the nonpositive deterministic
contribution gives \(f'(x_*)\leq-h^{-1}\), and hence \(f\notin D_h\).
Since the lower tail probability of a centered Gaussian variable at a fixed
negative threshold is nondecreasing in its variance,
\[
 \mathbb P(f\notin D_h\mid\mathcal F_t)
 \geq p_h.
\]
It follows that, on \(A_N\), the conditional probability of completing the
next time step is at most \(1-p_h\).  Since \(A_{N+1}\subseteq A_N\) and
\(A_N\in\mathcal F_t\), the tower property yields
\[
\begin{aligned}
 \mathbb P(A_{N+1})
 &=\mathbb E\!\left[
   \mathbb E[\mathbf 1_{A_{N+1}}\mid\mathcal F_t]\right]\\
 &=\mathbb E\!\left[\mathbf 1_{A_N}
   \mathbb P(A_{N+1}\mid\mathcal F_t)\right]\\
 &\leq(1-p_h)\mathbb P(A_N).
\end{aligned}
\]
Hence \(\mathbb P(A_N)\leq(1-p_h)^N\to0\).  The event of indefinite iteration
is \(\bigcap_{N\geq1}A_N\), which has probability zero.
\end{proof}

\begin{remark}[Scope of Theorem~\ref{thm:burgers-breakdown}]
Theorem~\ref{thm:burgers-breakdown} concerns the classical Sobolev Burgers
flow used in deterministic splitting analyses such as \cite{HKRT2011}.
The result concerns indefinite iteration at each fixed \(h>0\) and does not
exclude convergence on fixed finite time intervals as \(h\to0\).
\end{remark}

\begin{remark}[Strang splitting]
The same conclusion holds for both Strang compositions by an
analogous argument, assuming the same regularity and nondegeneracy
conditions on \(Z_{h/2}\).
\end{remark}

\section{Lyapunov structure and uniform moment bounds}
\label{sec:discrete-lyapunov}

In this section, we introduce two modified invariants that the KdV step
preserves exactly.  The splitting approximation therefore satisfies the same
Lyapunov inequalities as the exact solution, which yields moment bounds that
are uniform in the number of steps and in the step size.  These bounds are
used throughout Sections~4 and~5.

\subsection{Modified invariants and moment estimates}

The two functionals used throughout are obtained from the second and sixth
invariants of the KdV hierarchy by adding corrections that make them coercive;
we refer to them as the modified invariants.  Throughout this subsection we
work on \(H^6_0\) under the assumptions
\begin{equation}
 \Phi\in L_2(\mathfrak U,H^6_0),
 \qquad \lambda\geq1.
 \label{eq:lyapunov-assumptions}
\end{equation}
We use the KdV conservation laws and their coercive modifications from
\cite[Section~2.2 and (7.6)--(7.7)]{GHMR2024}.  Choose \(\bar\alpha_2>0\) sufficiently large
for \eqref{eq:I2-coercive} and \eqref{eq:G-semimartingale-bounds}, and set
\begin{align}
 I_2(v)&=\int_{\mathbb T}\left\{(\partial_x^2v)^2
 -\frac53v(\partial_xv)^2+\frac5{36}v^4\right\}\,\mathrm dx,
 \label{eq:I2-definition}\\
 I_2^+(v)&=I_2(v)+\bar\alpha_2(1+\|v\|_{L^2}^2)^{7/3}.
 \label{eq:I2plus-definition}
\end{align}
With this choice, \(I_2^+\geq1\) and
\begin{equation}
 c\left(\|\partial_x^2v\|_{L^2}^2+(1+\|v\|_{L^2}^2)^{7/3}\right)
 \leq I_2^+(v)
 \leq C\left(\|\partial_x^2v\|_{L^2}^2+(1+\|v\|_{L^2}^2)^{7/3}\right).
 \label{eq:I2-coercive}
\end{equation}
Let \(I_6\) denote the sixth invariant of the KdV hierarchy, written in the
form
\begin{equation}
 I_6(v)=\int_{\mathbb T}\left\{(\partial_x^6v)^2
       -\alpha_6v(\partial_x^5v)^2
       +\mathcal Q_6(v,\partial_xv,\ldots,\partial_x^4v)\right\}\,\mathrm dx,
 \label{eq:I6-structure}
\end{equation}
where \(\mathcal Q_6\) is the lower-order differential polynomial appearing in
that invariant; see the construction there.  Choose
\(\bar\alpha_6,\bar q_6\geq1\) sufficiently large so
that \eqref{eq:I6-coercive} and \eqref{eq:I6-coefficient-bounds} hold,
and define
\begin{equation}
 I_6^+(v):=I_6(v)+\bar\alpha_6
 (1+\|v\|_{L^2}^2)^{\bar q_6}.
 \label{eq:I6plus-definition}
\end{equation}
With this choice, \(I_6^+\geq1\) and
\begin{equation}
 c\left(\|v\|_{H^6}^2+(1+\|v\|_{L^2}^2)^{\bar q_6}\right)
 \leq I_6^+(v)
 \leq C\left(\|v\|_{H^6}^2+(1+\|v\|_{L^2}^2)^{\bar q_6}\right).
 \label{eq:I6-coercive}
\end{equation}
Both functionals are exactly invariant under the KdV flow:
\begin{equation}
 I_2^+(\mathcal K_tv)=I_2^+(v),
 \qquad I_6^+(\mathcal K_tv)=I_6^+(v).
 \label{eq:modified-invariance}
\end{equation}
These identities follow from the KdV conservation laws for smooth \(v\), and
extend by density and continuity to \(H^2\) for \(I_2^+\) and to \(H^6\) for
\(I_6^+\).  Since each KdV step leaves both invariants unchanged, the Lyapunov
estimate for the splitting approximation is carried entirely by the
Ornstein--Uhlenbeck steps.  In addition, the mean-zero assumption allows us to
apply the stochastic Lyapunov estimates of that reference.

Define the Lyapunov function
\begin{equation}
 G(v)=\bigl(I_2^+(v)\bigr)^{3/7}.
 \label{eq:G-definition}
\end{equation}
This power matches the spatial estimate needed for KdV stability; see also
\cite[Remark~7.4]{GHMR2024}.  Indeed,
\eqref{eq:I2-coercive} and the one-dimensional Gagliardo--Nirenberg inequality
give, for \(v\in H^2_0\),
\begin{equation}
 \|\partial_xv\|_{L^\infty}
 \leq C\|v\|_{L^2}^{1/4}\|\partial_x^2v\|_{L^2}^{3/4}
 \leq C\bigl(I_2^+(v)\bigr)^{3/56+3/8}=CG(v).
 \label{eq:G-spatial-bound}
\end{equation}
The exponent \(3/7\) is the smallest power of \(I_2^+\) that controls
\(\|\partial_xv\|_{L^\infty}\) through these interpolation and coercivity estimates.
Moreover, \(G(v)\) is comparable to
\(\|\partial_x^2v\|_{L^2}^{6/7}+1+\|v\|_{L^2}^2\).

For \(s\in(t_j,t_{j+1}]\), define the auxiliary piecewise
Ornstein--Uhlenbeck interpolation process of the splitting approximation by
\begin{equation}
 \widetilde U_h(s)=e^{-\lambda(s-t_j)}\mathcal K_hU^j
 +\int_{t_j}^{s}e^{-\lambda(s-r)}\Phi\,\mathrm dW(r).
 \label{eq:ou-interpolation}
\end{equation}
Thus \(\widetilde U_h(t_{j+1})=U^{j+1}\).  At each \(t_j\), the interpolation
jumps from \(U^j\) to \(\mathcal K_hU^j\).  By
\eqref{eq:modified-invariance}, neither modified invariant changes across this
jump.

Let \(\mathcal X\) denote either the exact solution
\(u\) or the splitting interpolation \(\widetilde U_h\).  We allow arbitrary
\(T\geq0\) in the exact case and take \(T=nh\) in the splitting case, so
that \(\mathcal X(T)=U^n\).  We write \(\mathcal X(0)=u(0)\) or \(U^0\),
respectively.

\begin{proposition}[Uniform bounds for the modified invariants]
\label{prop:uniform-lyapunov}
Assume \eqref{eq:lyapunov-assumptions}, and let \(\mathcal X(0)\) be an
\(H^6_0\)-valued, \(\mathcal F_0\)-measurable random variable.

For every \(p>0\), there is
\(C_p=C_p(\Phi)<\infty\), independent of \(h,T\), \(\lambda\geq1\), and
the initial state, such that
\begin{equation}
 \mathbb E\!\left[\bigl(I_6^+(\mathcal X(T))\bigr)^p
       \middle|\mathcal F_0\right]
 \leq e^{-\lambda pT}\bigl(I_6^+(\mathcal X(0))\bigr)^p+C_p.
 \label{eq:common-polynomial-bound}
\end{equation}
There are constants
\(\eta_*=\eta_*(\|\Phi\|_{L_2(\mathfrak U,H^2)})>0\),
\(C_0=C_0(\Phi)<\infty\), and \(C_E=C_E(\Phi)<\infty\), independent of
\(h\), \(T\), \(\lambda\geq1\), and the initial state, such that, for
\(0\leq\eta\leq\eta_*\),
\begin{align}
 \mathbb E\!\left[\exp\!\left\{\eta G(\mathcal X(T))
       +\frac{3\eta\lambda}{7}\int_0^T G(\mathcal X(s))\,\mathrm ds\right\}
       \middle|\mathcal F_0\right]
 &\leq2\exp\{\eta(G(\mathcal X(0))+C_0\lambda T)\},
 \label{eq:common-integrated-exp}\\
 \mathbb E\!\left[e^{\eta G(\mathcal X(T))}\middle|\mathcal F_0\right]
 &\leq C_E\exp\{\eta e^{-3\lambda T/7}G(\mathcal X(0))\}.
 \label{eq:common-endpoint-exp}
\end{align}
In the splitting case the integral in \eqref{eq:common-integrated-exp} is over
the interpolation \eqref{eq:ou-interpolation}.  The corresponding conditional
estimates hold after restarting at any fixed time \(s\) in the exact
case and at any \(t_j\) in the splitting case, with conditioning on
\(\mathcal F_s\) and \(\mathcal F_{t_j}\), respectively.
\end{proposition}

\begin{proof}
\medskip\noindent\emph{Step 1: stopped process and It\^o's formula.} Let
\(\tau_R\) be the first exit time of \(\mathcal X\) from the \(H^6\)-ball of
radius \(R\). We first work up to \(\tau_R\). The regularization argument in
\cite[Section~3.1]{GHMR2024} justifies the evolution formula below for the
exact equation with \(H^6_0\) initial data and
\(\Phi\in L_2(\mathfrak U,H^6_0)\). The KdV flow preserves \(I_6\) and its
\(L^2\)-correction, so
its contribution to the evolution of \(I_6^+\) vanishes.  On each interval
\((t_j,t_{j+1}]\), the interpolation process satisfies
\[
 \mathrm d\mathcal X=-\lambda\mathcal X\,\mathrm ds+\Phi\,\mathrm dW.
\]
This is an \(H^6\)-valued semimartingale, so It\^o's formula applies
directly on each Ornstein--Uhlenbeck interval.  Both cases therefore give the
same scalar evolution formula:
\[
 \mathrm dI_6^+(\mathcal X)+2\lambda I_6^+(\mathcal X)\,\mathrm ds
 =\mathcal A_6(\mathcal X)\,\mathrm ds+\mathcal B_6(\mathcal X)\,\mathrm dW.
\]
For \(v\in H^6_0\), the coefficients are
\[
 \begin{split}
 \mathcal A_6(v)
 &=-\lambda DI_6^+(v)[v]+2\lambda I_6^+(v)
   +\frac12\sum_kD^2I_6^+(v)[\Phi e_k,\Phi e_k],\\
 \mathcal B_6(v)&=(DI_6^+(v)[\Phi e_k])_k.
 \end{split}
\]
Here \((e_k)\) is an orthonormal basis of \(\mathfrak U\), and
\(DI_6^+(v)[w]\) denotes the Fr\'echet derivative of \(I_6^+\) at \(v\),
applied to the direction \(w\).  Similarly,
\(D^2I_6^+(v)[w_1,w_2]\) denotes its second Fr\'echet derivative applied to
the two directions \(w_1,w_2\).  In particular,
\(-\lambda DI_6^+(v)[v]\) is the contribution of the damping term.
The scalar evolution formula avoids treating the full exact equation as an
\(H^6\)-valued semimartingale, since \(F(\mathcal X)\) is in general only
\(H^3\)-valued.

The coefficient estimates are obtained using integration by parts,
one-dimensional Gagliardo--Nirenberg interpolation, and Young's inequality;
the correction in \(I_6^+\) controls the remaining powers of
\(\|v\|_{L^2}\). Using the corresponding coefficient estimates from
\cite[Lemma~3.5 and the proof of Theorem~3.1]{GHMR2024} with \(m=6\), and
using \(\lambda\geq1\), we obtain, for every \(\varepsilon>0\),
\begin{equation}
 \mathcal A_6(v)\leq\varepsilon\lambda I_6^+(v)+C_\varepsilon\lambda,
 \qquad
 \|\mathcal B_6(v)\|_{L_2(\mathfrak U,\mathbb R)}^2
 \leq\varepsilon\lambda\bigl(I_6^+(v)\bigr)^2+C_\varepsilon\lambda,
 \label{eq:I6-coefficient-bounds}
\end{equation}
where \(C_\varepsilon\) is independent of \(\lambda\geq1\).
In the splitting case, \eqref{eq:modified-invariance} ensures that both
modified invariants are continuous across the KdV steps.  Their scalar
evolution formulas therefore combine across the intervals without an
additional term.

\medskip\noindent\emph{Step 2: polynomial moment bounds.} Since
\(I_6^+\geq1\), It\^o's formula applies to \((I_6^+)^p\) for every real
\(p>0\).  For \(0<p\leq1\), the second-order term is nonpositive and can be
discarded in the upper bound.  For \(p>1\), the coefficient estimates
\eqref{eq:I6-coefficient-bounds} and Young's inequality absorb the lower
powers into the damping term.  Choosing \(\varepsilon\) according to \(p\), we
obtain
\[
 \mathrm d\bigl(I_6^+(\mathcal X)\bigr)^p
 +\lambda p\bigl(I_6^+(\mathcal X)\bigr)^p\,\mathrm ds
 \leq C_p\lambda\,\mathrm ds+\mathrm dM_p(s),
\]
where \(M_p\) is a continuous local martingale.
Applying the integrating factor up to \(T\wedge\tau_R\) and taking
conditional expectation gives
\[
 \begin{split}
 &\mathbb E\!\left[
 e^{\lambda p(T\wedge\tau_R)}
 \bigl(I_6^+(\mathcal X(T\wedge\tau_R))\bigr)^p
 \middle|\mathcal F_0\right]\\
 &\qquad\leq\bigl(I_6^+(\mathcal X(0))\bigr)^p
       +C_p\lambda\int_0^T e^{\lambda ps}\,\mathrm ds.
 \end{split}
\]
Global existence allows \(R\to\infty\).  Conditional Fatou's lemma, followed
by multiplication by \(e^{-\lambda pT}\), proves
\eqref{eq:common-polynomial-bound}: the source term becomes
\((C_p/p)(1-e^{-\lambda pT})\), which is bounded uniformly in \(T\) and in
\(\lambda\geq1\).

\medskip\noindent\emph{Step 3: exponential moment bounds.}
For \(I_2^+\), the coefficient estimates in \cite[Lemma~7.3]{GHMR2024}
bound the drift correction above by \((\lambda/100)I_2^++C\lambda\), while
the full noise coefficient satisfies
\[
 \left(\sum_k|DI_2^+(v)[\Phi e_k]|^2\right)^{1/2}
 \leq C\bigl(I_2^+(v)\bigr)^{11/14}.
\]
Here \(C\) depends on \(\|\Phi\|_{L_2(\mathfrak U,H^2)}\).
We apply It\^o's formula to \(G=(I_2^+)^{3/7}\). The chain rule
\(DG=\frac37(I_2^+)^{-4/7}DI_2^+\) therefore gives
\[
 \left(\sum_k|DG(v)[\Phi e_k]|^2\right)^{1/2}
 \leq C\bigl(I_2^+(v)\bigr)^{-4/7+11/14}
 =C\bigl(I_2^+(v)\bigr)^{3/14}.
\]
The resulting scalar evolution and coefficient bounds are
\begin{equation}
 \begin{aligned}
 \mathrm dG(\mathcal X(s))&=\mathfrak b(s)\,\mathrm ds+\mathrm dM(s),\\
 \mathfrak b(s)&\leq-\frac{5\lambda}{7}G(\mathcal X(s))+C_0\lambda,\\
 |\mathfrak b(s)|&\leq C\lambda(1+G(\mathcal X(s))),\\
 \mathrm d[M](s)&=\sum_k|DG(\mathcal X(s))[\Phi e_k]|^2\,\mathrm ds
       \leq C_{\mathrm{qv}}G(\mathcal X(s))\,\mathrm ds.
 \end{aligned}
 \label{eq:G-semimartingale-bounds}
\end{equation}
Here \(M\) is a continuous local martingale and \([M]\) denotes its
quadratic variation.  The upper bound on \(\mathfrak b\) uses the
nonpositive second-order term in It\^o's formula for \((I_2^+)^{3/7}\).
The bound on \(|\mathfrak b|\) follows by estimating each drift term in
absolute value, using the same coefficient formulas.

Choose \(\theta_0>0\) so that \(\theta_0C_{\mathrm{qv}}/2\leq2/7\), and set
\(\eta_*=\theta_0/2\).  Integrating the scalar inequality and retaining
\(3\lambda/7\) of the damping on the left gives
\[
 \begin{split}
 G(\mathcal X(T))+\frac{3\lambda}{7}\int_0^T G(\mathcal X(s))\,\mathrm ds
 \leq{}&G(\mathcal X(0))+C_0\lambda T\\
 &+M(T)-\frac{2\lambda}{7}\int_0^T G(\mathcal X(s))\,\mathrm ds.
 \end{split}
\]
Since \(\lambda\geq1\), the last term absorbs
\(\theta_0[M](T)/2\).  The exponential martingale inequality at parameter
\(\theta_0\), followed by integration of its tail bound, proves
\eqref{eq:common-integrated-exp}; the factor is at most
\(1+\eta/(\theta_0-\eta)\leq2\) for \(0\leq\eta\leq\eta_*\).

For the endpoint exponential-moment estimate, fix \(T\) and introduce the
continuous local martingale
\[
 \widehat M_T(r)=\int_0^r e^{-3\lambda(T-s)/7}\,\mathrm dM(s),
 \qquad 0\leq r\leq T.
\]
Multiplying the scalar inequality by
\(e^{-3\lambda(T-s)/7}\) and integrating gives
\[
 \begin{split}
 &G(\mathcal X(T))+\frac{2\lambda}{7}\int_0^T
 e^{-3\lambda(T-s)/7}G(\mathcal X(s))\,\mathrm ds\\
 &\qquad\leq e^{-3\lambda T/7}G(\mathcal X(0))
 +C_0\lambda\int_0^T e^{-3\lambda(T-s)/7}\,\mathrm ds+\widehat M_T(T).
 \end{split}
\]
The deterministic contribution is bounded independently of \(T\) and
\(\lambda\geq1\), since
\[
 \lambda\int_0^T e^{-3\lambda(T-s)/7}\,\mathrm ds
 =\frac73(1-e^{-3\lambda T/7})\leq\frac73.
\]
Moreover,
\[
 \begin{split}
 [\widehat M_T](T)
 &=\int_0^T e^{-6\lambda(T-s)/7}\,\mathrm d[M](s)\\
 &\leq C_{\mathrm{qv}}\int_0^T
 e^{-3\lambda(T-s)/7}G(\mathcal X(s))\,\mathrm ds,
 \end{split}
\]
where we used \(e^{-6\lambda(T-s)/7}\leq e^{-3\lambda(T-s)/7}\).
The positive integral term again absorbs the quadratic-variation term in
the exponential martingale inequality.  Integrating the resulting tail
bound for \(0\leq\eta\leq\eta_*\) yields
\eqref{eq:common-endpoint-exp}, with \(C_E\) independent of \(T\) and
\(\lambda\geq1\).

These exponential estimates are first obtained for stopped scalar processes;
conditional Fatou's lemma removes the stopping times.  In the splitting case,
each KdV step leaves the invariants unchanged, so the scalar estimates cover
the whole interpolation process.  Finally, the same arguments apply after the
stated restart times, since the future Wiener increments are independent of
the conditioning sigma-algebra.
\end{proof}

Combining \eqref{eq:common-polynomial-bound} and
\eqref{eq:common-endpoint-exp} yields uniform moment bounds with exponential weights.

\begin{corollary}[Uniform moment bounds with exponential weights]
\label{cor:mixed-moments}
Let \(q\geq0\), \(0<\eta<\eta_*\), and choose H\"older-conjugate exponents
\(p,p'>1\) such that \(\eta p'\leq\eta_*\).  If
\[
 \mathbb E\left[\bigl(I_6^+(U^0)\bigr)^{qp}\right]<\infty,
 \qquad \mathbb E[e^{\eta p'G(U^0)}]<\infty,
\]
then
\begin{equation}
 \begin{aligned}
 &\sup_{h>0}\sup_{n\geq0}
 \mathbb E\left[e^{\eta G(U^n)}
       \bigl(1+\bigl(I_6^+(U^n)\bigr)^q\bigr)\right]<\infty,\\
 &\sup_{h>0}\sup_{n\geq0}
 \mathbb E\left[e^{\eta G(U^n)}(1+\|U^n\|_{H^6}^{2q})\right]<\infty.
 \end{aligned}
 \label{eq:mixed-uniform}
\end{equation}
\end{corollary}

\begin{proof}
For \(q>0\), H\"older's inequality gives
\[
 \begin{split}
 &\mathbb E\left[e^{\eta G(U^n)}
       \bigl(1+\bigl(I_6^+(U^n)\bigr)^q\bigr)\right]\\
 &\qquad\leq
 \left(\mathbb E\left[\bigl(1+\bigl(I_6^+(U^n)\bigr)^q\bigr)^p\right]\right)^{1/p}
 \left(\mathbb E[e^{\eta p'G(U^n)}]\right)^{1/p'}.
 \end{split}
\]
The first factor is controlled by \eqref{eq:common-polynomial-bound} with
moment order \(qp\), and the second by \eqref{eq:common-endpoint-exp}, since
\(\eta p'\leq\eta_*\).  The initial moment assumptions make both bounds
uniform in \(h\) and \(n\).  The case \(q=0\), where the polynomial factor is
constant, requires only \eqref{eq:common-endpoint-exp}.  The second bound in
\eqref{eq:mixed-uniform} follows from \eqref{eq:I6-coercive}.
\end{proof}
\section{Strong approximation}
\label{sec:strong-convergence}

In this section, we compare the exact solution with the splitting
approximation.  Section~4.1 estimates the local error, decomposing it into its conditional
mean and its zero-mean part; the two enter the
global estimates in different ways.  Section~4.2 then proves strong order
\(1\) on every fixed time interval, and Section~4.3 establishes the same order
uniformly in time.

\subsection{Local error estimates}
\label{sec:local-error}

Recall the one-step map \eqref{eq:lie-map}.  Fix a deterministic initial state
\(x\), and drive the exact solution and the splitting approximation over one
time step by the same Wiener process.  Set
\[
 X_s=\mathcal S^\omega_{s,0}x,
 \qquad V_s=\mathcal K_sx,
 \qquad
 \delta_h(x)=\mathcal S^\omega_{h,0}x-\mathcal L^\omega_{h,0}x.
\]
Here \(\mathbb E\) denotes expectation over the Wiener increments, with
the initial state \(x\) fixed.  Decompose the local error into its mean and its zero-mean part:
\[
 \bar\delta_h(x)=\mathbb E[\delta_h(x)],
 \qquad \delta_h^\circ(x)=\delta_h(x)-\bar\delta_h(x).
\]
For an \(\mathcal F_t\)-measurable, \(H^6_0\)-valued random initial state
\(\chi\), define the corresponding local error by
\[
 \delta_{h,t}(\chi)=\mathcal S^\omega_{t+h,t}\chi
                    -\mathcal L^\omega_{t+h,t}\chi.
\]
Whenever \(\delta_{h,t}(\chi)\) is integrable, set
\[
 \bar\delta_{h,t}(\chi)
 =\mathbb E[\delta_{h,t}(\chi)\mid\mathcal F_t],
 \qquad
 \delta_{h,t}^\circ(\chi)
 =\delta_{h,t}(\chi)-\bar\delta_{h,t}(\chi).
\]

Throughout the remainder of this subsection, assume
\begin{equation}
 x\in H^6_0,
 \qquad \Phi\in L_2(\mathfrak U,H^6_0),
 \qquad \lambda\geq1,
 \qquad 0<h\leq h_0.
 \label{eq:local-moment-assumptions}
\end{equation}
The required \(H^6\) moment bounds on finite time intervals for the exact
solution follow from Proposition~\ref{prop:uniform-lyapunov}.  Constants may
depend on the fixed \(h_0,\lambda,\Phi\), and the indicated moment order, but
not on \(x\) or \(h\).

We use the elementary estimates
\begin{align}
 \|F(v)-F(w)\|_{L^2}
 &\leq C(1+\|v\|_{H^1}+\|w\|_{H^1})\|v-w\|_{H^3},
 \label{eq:F-local-Lipschitz}\\
 \|F(v)\|_{H^3}&\leq C(1+\|v\|_{H^6}^2).
 \label{eq:F-H6-H3}
\end{align}
In \eqref{eq:F-H6-H3}, the Airy term \(\partial_x^3v\) is controlled in
\(H^3\) by the \(H^6\) norm, and the quadratic term by the one-dimensional
Sobolev algebra property, with no regularity beyond \(H^6\) needed.

\begin{proposition}[Local error estimates]
\label{prop:local-error-moments}
Assume \eqref{eq:local-moment-assumptions}.
For every integer \(p\geq1\), there exist \(q_p\geq0\) and
\(C_p<\infty\) such that
\begin{equation}
 \mathbb E[\|\delta_h(x)\|_{L^2}^{2p}]
 \leq C_ph^{3p}(1+I_6^+(x))^{q_p}.
 \label{eq:local-error-moment}
\end{equation}
In particular,
\begin{align}
 \mathbb E[\|\delta_h(x)\|_{L^2}^{2}]
 &\leq Ch^{3}(1+(I_6^+(x))^3),
 \label{eq:local-error-second-moment}\\
 \mathbb E[\|\delta_h(x)\|_{L^2}^{4}]
 &\leq Ch^{6}(1+(I_6^+(x))^6).
 \label{eq:local-error-fourth-moment}
\end{align}
There exists \(q\geq0\) such that
\begin{equation}
 \|\bar\delta_h(x)\|_{L^2}
 \leq Ch^2(1+I_6^+(x))^q.
 \label{eq:local-error-mean}
\end{equation}
Moreover, \(\mathbb E[\delta_h^\circ(x)]=0\) and
\begin{equation}
 \mathbb E[\|\delta_h^\circ(x)\|_{L^2}^{2p}]
 \leq C_ph^{3p}(1+I_6^+(x))^{q_p}.
 \label{eq:zero-mean-local-error}
\end{equation}
The mean estimate and the moment bounds for the zero-mean part also hold
conditionally.  Let \(\chi\) be an
\(H^6_0\)-valued, \(\mathcal F_t\)-measurable random state satisfying
\begin{equation}
 \mathbb E[(1+I_6^+(\chi))^{q+q_p}]<\infty.
 \label{eq:conditional-moment-assumption}
\end{equation}
Then, almost surely,
\begin{align}
 \|\bar\delta_{h,t}(\chi)\|_{L^2}
 &\leq Ch^2(1+I_6^+(\chi))^q,
 \label{eq:conditional-local-mean}\\
 \mathbb E[\delta_{h,t}^\circ(\chi)\mid\mathcal F_t]&=0,
 \label{eq:conditional-local-centering}\\
 \mathbb E[\|\delta_{h,t}^\circ(\chi)\|_{L^2}^{2p}
             \mid\mathcal F_t]
 &\leq C_ph^{3p}(1+I_6^+(\chi))^{q_p}.
 \label{eq:conditional-local-moment}
\end{align}
\end{proposition}

\begin{proof}
The variation-of-constants formula for \eqref{eq:skdv} and the integral
formulation of the KdV flow give
\begin{equation}
 \begin{split}
 \delta_h(x)
 ={}&\int_0^h e^{-\lambda(h-s)}\{F(X_s)-F(V_s)\}\,\mathrm ds\\
 &+\int_0^h\{e^{-\lambda(h-s)}-e^{-\lambda h}\}F(V_s)\,\mathrm ds.
 \end{split}
 \label{eq:local-error-identity}
\end{equation}
The stochastic convolutions in the exact and splitting maps cancel.  For
\(H^6\) initial data, this identity holds in \(H^3\), and hence in \(L^2\).

Fix \(p\geq1\), and enlarge \(q_p\) as needed so that the polynomial moment
factors in these estimates are controlled by \((1+I_6^+(x))^{q_p}\).  In
\(H^3\), the increment formulas are

\begin{align}
 X_s-x&=\int_0^s\Phi\,\mathrm dW(r)
       +\int_0^s\{F(X_r)-\lambda X_r\}\,\mathrm dr,
 \label{eq:exact-H3-increment}\\
 V_s-x&=\int_0^sF(V_r)\,\mathrm dr.
 \label{eq:kdv-H3-increment}
\end{align}
We subtract \eqref{eq:kdv-H3-increment} from \eqref{eq:exact-H3-increment} and
estimate the two resulting terms.  The Burkholder--Davis--Gundy inequality
bounds the \(2p\)-moment of the stochastic integral by \(C_ps^p\).  Jensen's
inequality, \eqref{eq:F-H6-H3}, the \(H^6\) moment bounds of
Proposition~\ref{prop:uniform-lyapunov}, and \(I_6^+(V_s)=I_6^+(x)\) bound
that of the drift integrals by \(C_ps^{2p}(1+I_6^+(x))^{q_p}\).  Consequently,
\begin{equation}
 \mathbb E[\|X_s-V_s\|_{H^3}^{2p}]
 \leq C_ps^p(1+I_6^+(x))^{q_p}.
 \label{eq:X-V-increment}
\end{equation}
Applying H\"older's inequality to \eqref{eq:F-local-Lipschitz}, with
\eqref{eq:X-V-increment} at order \(4p\) and the \(H^1\) moment bounds, gives
\begin{equation}
 \mathbb E[\|F(X_s)-F(V_s)\|_{L^2}^{2p}]
 \leq C_ps^p(1+I_6^+(x))^{q_p}.
 \label{eq:F-difference-moment}
\end{equation}
Jensen's inequality and \eqref{eq:F-difference-moment} bound the
\(2p\)-moment of the first integral in \eqref{eq:local-error-identity} by
\[
 C_p h^{2p-1}\int_0^h s^p\,\mathrm ds\,(1+I_6^+(x))^{q_p}
 \leq C_ph^{3p}(1+I_6^+(x))^{q_p}.
\]
For the second integral, the estimate
\[
 |e^{-\lambda(h-s)}-e^{-\lambda h}|\leq\lambda s
\]
and the moment bound on \(F(V_s)\) give a \(2p\)-moment bounded by
\(C_ph^{4p}(1+I_6^+(x))^{q_p}\).  Combining the two bounds proves
\eqref{eq:local-error-moment}.  Tracking the polynomial factors for \(p=1\)
and \(p=2\) gives \eqref{eq:local-error-second-moment} and
\eqref{eq:local-error-fourth-moment}, respectively.

To prove \eqref{eq:local-error-mean}, enlarge
\(q\) if necessary and set \(\Delta_s=X_s-V_s\).  Subtracting
\eqref{eq:kdv-H3-increment} from \eqref{eq:exact-H3-increment} and taking
expectation removes the stochastic integral.  The drift moment bounds give
\begin{equation}
 \|\mathbb E[\Delta_s]\|_{H^3}\leq Cs(1+I_6^+(x))^q.
 \label{eq:mean-Delta}
\end{equation}
Since \(F\) is quadratic,
\begin{equation}
 F(V_s+\Delta_s)-F(V_s)
 =-\partial_x^3\Delta_s-\partial_x(V_s\Delta_s)
  -\frac12\partial_x(\Delta_s^2).
 \label{eq:quadratic-expansion}
\end{equation}
Since \(V_s\) is deterministic,
\(\mathbb E[V_s\Delta_s]=V_s\mathbb E[\Delta_s]\).  Thus the expectations
of the first two terms are \(O(s)\) in \(L^2\) by \eqref{eq:mean-Delta}
and the \(H^1\) bound on \(V_s\).  For the last one,
\[
 \left\|\partial_x\mathbb E[\Delta_s^2]\right\|_{L^2}
 \leq C\mathbb E[\|\Delta_s\|_{H^1}^2]
 \leq Cs(1+I_6^+(x))^q.
\]
Hence \(\|\mathbb E[F(X_s)-F(V_s)]\|_{L^2}\leq Cs(1+I_6^+(x))^q\).
Taking expectations in \eqref{eq:local-error-identity}, we estimate the first
integral in \(L^2\) using the bound for
\(\mathbb E[F(X_s)-F(V_s)]\), which gives
\[
 C(1+I_6^+(x))^q\int_0^h s\,\mathrm ds.
\]
For the second integral,
\(\lvert e^{-\lambda(h-s)}-e^{-\lambda h}\rvert\leq\lambda s\)
and the bound on \(F(V_s)\) give the same estimate.  Since
\(\int_0^h s\,\mathrm ds=h^2/2\), \eqref{eq:local-error-mean} follows.
The definition of \(\delta_h^\circ\), Jensen's
inequality, and \eqref{eq:local-error-moment} give
\eqref{eq:zero-mean-local-error}.  Conditionally on \(\mathcal F_t\), the
initial state \(\chi\) is fixed and the Wiener increments after time \(t\) are
independent of \(\mathcal F_t\).  By time homogeneity,
\(\bar\delta_{h,t}(\chi)\) therefore agrees almost surely with the
deterministic function \(\bar\delta_h\) evaluated at \(\chi\).  Applying the
same moment estimates with initial state \(\chi\) proves the conditional
bounds.
\end{proof}

Following \cite[(7.12)]{GHMR2024}, for \(a>0\), define the weighted
distance-like function on \(H^2_0\times H^2_0\) by
\begin{equation}
 d_a(y,z)=e^{a\{G(y)+G(z)\}}\|y-z\|_{L^2}.
 \label{eq:weighted-cost-a}
\end{equation}
It is continuous and symmetric, and vanishes if and only if \(y=z\), but is
not assumed to satisfy the triangle inequality.  It is therefore a distance-like function in
the sense of \cite{HairerMattinglyScheutzow2011}.  The following bounds
control the local error in the weighted strong and weak estimates.

\begin{lemma}[Local error estimates with exponential weights]
\label{lem:weighted-local}
Assume \eqref{eq:local-moment-assumptions}.  Let \(\eta_*\) be as in
Proposition~\ref{prop:uniform-lyapunov}.  If \(0<a\leq\eta_*/4\), then
\begin{align}
 &\mathbb E\left[
 e^{a\{G(\mathcal S^\omega_{h,0}x)+G(\mathcal L^\omega_{h,0}x)\}}
 \|\delta_h(x)\|_{L^2}^{2}\right]
 \leq Ch^3e^{2aG(x)}(1+(I_6^+(x))^3),
 \label{eq:weighted-local-second}\\
 &\mathbb E[d_a(\mathcal S^\omega_{h,0}x,\mathcal L^\omega_{h,0}x)]
 \leq Ch^{3/2}e^{2aG(x)}(1+(I_6^+(x))^{3/2}).
 \label{eq:weighted-local-first}
\end{align}
The constants are independent of \(x\) and uniform for \(0<h\leq h_0\);
they may depend on \(a,h_0,\lambda\), and \(\Phi\).  Both estimates also hold
conditionally.  For the random state \(\chi\) in
Proposition~\ref{prop:local-error-moments}, almost surely,
\begin{align*}
 &\mathbb E\left[
 e^{a\{G(\mathcal S^\omega_{t+h,t}\chi)
       +G(\mathcal L^\omega_{t+h,t}\chi)\}}
 \|\delta_{h,t}(\chi)\|_{L^2}^{2}
 \mathrel{\Big|}\mathcal F_t\right]
 \leq Ch^3e^{2aG(\chi)}(1+(I_6^+(\chi))^3),\\
 &\mathbb E\left[
 d_a(\mathcal S^\omega_{t+h,t}\chi,
     \mathcal L^\omega_{t+h,t}\chi)
 \mathrel{\Big|}\mathcal F_t\right]
 \leq Ch^{3/2}e^{2aG(\chi)}(1+(I_6^+(\chi))^{3/2}).
\end{align*}
\end{lemma}

\begin{proof}
H\"older's inequality with exponents \(2,4,4\) yields
\[
 \begin{split}
 &\mathbb E\left[e^{a(G(\mathcal S^\omega_{h,0}x)
 +G(\mathcal L^\omega_{h,0}x))}
 \|\delta_h(x)\|_{L^2}^2\right]\\
 &\quad\leq
 (\mathbb E[\|\delta_h(x)\|_{L^2}^4])^{1/2}
 (\mathbb E[e^{4aG(\mathcal S^\omega_{h,0}x)}])^{1/4}
 (\mathbb E[e^{4aG(\mathcal L^\omega_{h,0}x)}])^{1/4}.
 \end{split}
\]
Applying Proposition~\ref{prop:uniform-lyapunov} to the exact and splitting
endpoints bounds the product of the last two factors by \(Ce^{2aG(x)}\).  The
fourth moment bound \eqref{eq:local-error-fourth-moment} proves
\eqref{eq:weighted-local-second}.  Using
\eqref{eq:local-error-second-moment} instead gives
\eqref{eq:weighted-local-first}.  Conditioning on
\(\mathcal F_t\) and applying the same H\"older estimate proves the
conditional versions.
\end{proof}

\subsection{Strong convergence on finite time intervals}
\label{sec:finite-strong}

Estimates \eqref{eq:local-error-mean} and \eqref{eq:zero-mean-local-error}
bound the mean of the local error by \(O(h^2)\) and its zero-mean part by
\(O(h^{3/2})\) in the root-mean-square norm. These are the local orders
required for strong order \(1\) in Milstein's mean-square convergence theorem
\cite{MilsteinTretyakov2004}. The Burgers nonlinearity, however, loses one
spatial derivative and hence is not locally Lipschitz in \(L^2\), so the
standard mean-square convergence theory does not apply directly. We overcome
this difficulty by combining the continuous dependence estimate with the
exponential Lyapunov bounds and separate estimates for the conditional mean
and zero-mean part of the local error. This proves first-order strong
convergence.

We first establish the continuous dependence estimate used in both the
finite-time and uniform-in-time proofs. Here, ``one-sided'' refers to the
dependence of the growth factor on only one of the two initial states.

\begin{lemma}[One-sided continuous dependence estimate for the KdV flow]
\label{lem:one-sided-kdv-stability}
There is a deterministic constant \(C_{\mathrm K}>0\) such that, for
\(x,y\in H^6_0\) and \(0<h\leq1\),
\begin{equation}
 \|\mathcal K_hx-\mathcal K_hy\|_{L^2}^2
 \leq e^{C_{\mathrm K}h(1+G(y))}\|x-y\|_{L^2}^2.
 \label{eq:one-sided-kdv-stability}
\end{equation}
\end{lemma}

\begin{proof}
Let \(v_s=\mathcal K_sx\), \(\widetilde v_s=\mathcal K_sy\), and
\(w_s=v_s-\widetilde v_s\).  The difference energy identity can be
written in terms of the second trajectory, since
\[
 \int_{\mathbb T}\{\partial_xv_s-\partial_x\widetilde v_s\}w_s^2\,\mathrm dx
 =\int_{\mathbb T}(\partial_xw_s)w_s^2\,\mathrm dx=0.
\]
Integration by parts gives
\[
 \frac{\mathrm d}{\mathrm ds}\|w_s\|_{L^2}^2
 =-\int_{\mathbb T}(\partial_x\widetilde v_s)w_s^2\,\mathrm dx
 \leq\|\partial_x\widetilde v_s\|_{L^\infty}\|w_s\|_{L^2}^2.
\]
Estimate \eqref{eq:G-spatial-bound} and KdV conservation give
\[
 \|\partial_x\widetilde v_s\|_{L^\infty}
 \leq CG(\widetilde v_s)=CG(y).
\]
Gronwall's inequality proves \eqref{eq:one-sided-kdv-stability}.
\end{proof}

\begin{theorem}[Finite-time strong convergence]
\label{thm:finite-strong-order-one}
Assume \(u_0\in H^6_0\),
\(\Phi\in L_2(\mathfrak U,H^6_0)\), and that the exact solution and the
splitting approximation satisfy \(u(0)=U^0=u_0\) and are driven by the same
Wiener process.  There is
\(\lambda_{\mathrm{fin}}=\lambda_{\mathrm{fin}}(\Phi)\geq1\) such that,
for every fixed \(\lambda\geq\lambda_{\mathrm{fin}}\), there is
\(h_0=h_0(\lambda,\Phi)>0\) for which
\begin{equation}
 \max_{0\leq nh\leq T}
 \mathbb E[\|u(t_n)-U^n\|_{L^2}^2]\leq C_Th^2,
 \qquad 0<h\leq h_0,
 \label{eq:finite-strong-order-one}
\end{equation}
for every \(T<\infty\).  The constant may depend on
\(T,u_0,\lambda,\Phi\), but not on \(h\) or \(n\) with \(nh\leq T\).
\end{theorem}

\begin{proof}
Write \(\mathbb E_n[\,\cdot\,]=
\mathbb E[\,\cdot\mid\mathcal F_{t_n}]\) and
\(e_n=u(t_n)-U^n\).  Then \(e_0=0\).  Set
\[
 p_n=\mathcal L^\omega_{t_{n+1},t_n}u(t_n)
     -\mathcal L^\omega_{t_{n+1},t_n}U^n,
 \qquad
 \delta_{n+1}=\delta_{h,t_n}(u(t_n)),
\]
\[
 \bar\delta_{n+1}=\mathbb E_n[\delta_{n+1}],
 \qquad
 \delta_{n+1}^\circ=\delta_{n+1}-\bar\delta_{n+1}.
\]
The one-step error decomposition is
\begin{equation}
 e_{n+1}=p_n+\delta_{n+1}
 =p_n+\bar\delta_{n+1}+\delta_{n+1}^\circ.
 \label{eq:finite-error-decomposition}
\end{equation}

\medskip\noindent\emph{Step 1: conditional local error estimates.}
Fix an integer \(p>1\).  After increasing \(q_p\) if necessary,
Proposition~\ref{prop:local-error-moments} gives
\begin{align}
 \|\bar\delta_{n+1}\|_{L^2}
 &\leq Ch^2(1+I_6^+(u(t_n)))^{q_p},
 \label{eq:finite-local-mean}\\
 \mathbb E_n[\|\delta_{n+1}^\circ\|_{L^2}^{2p}]
 &\leq Ch^{3p}(1+I_6^+(u(t_n)))^{2pq_p}.
 \label{eq:finite-local-centered-moment}
\end{align}
Conditional Jensen's inequality also gives
\begin{equation}
 \mathbb E_n[\|\delta_{n+1}^\circ\|_{L^2}^2]
 \leq\bigl(\mathbb E_n[\|\delta_{n+1}^\circ\|_{L^2}^{2p}]\bigr)^{1/p}
 \leq Ch^3(1+I_6^+(u(t_n)))^{2q_p}.
 \label{eq:finite-local-centered-two}
\end{equation}
The polynomial Lyapunov estimate in Proposition~\ref{prop:uniform-lyapunov}
implies
\begin{equation}
 \sup_{0<h\leq h_0}\sup_{nh\leq T}
 \mathbb E[(1+I_6^+(u(t_n)))^{2pq_p}]<\infty.
 \label{eq:finite-local-source-moment}
\end{equation}

\medskip\noindent\emph{Step 2: one-step pathwise stability.}
Applying Lemma~\ref{lem:one-sided-kdv-stability} with \(x=U^n\) and
\(y=u(t_n)\) puts the coefficient on the exact solution.  Since the stochastic
convolution cancels in the difference between the two splitting updates,
\[
 p_n=e^{-\lambda h}\{\mathcal K_hu(t_n)-\mathcal K_hU^n\}
\]
is \(\mathcal F_{t_n}\)-measurable and satisfies
\begin{equation}
 \|p_n\|_{L^2}^2
 \leq e^{-2\lambda h}e^{C_{\mathrm K}h(1+G(u(t_n)))}\|e_n\|_{L^2}^2,
 \label{eq:finite-propagated-stability}
\end{equation}
where \(C_{\mathrm K}\) is the deterministic constant of
Lemma~\ref{lem:one-sided-kdv-stability}, independent of \(n\), \(h\), and the
initial state.

\medskip\noindent\emph{Step 3: exponential integrability of the discrete
Lyapunov sum.}
The product of the growth factors in
\eqref{eq:finite-propagated-stability} involves
\(h\sum_jG(u(t_j))\) in its exponent.
For the exact solution, \eqref{eq:G-semimartingale-bounds} gives
\[
 \mathrm dG(u(t))=\mathfrak b(t)\,\mathrm dt+\mathrm dM(t),\qquad
 |\mathfrak b(t)|\leq C\lambda(1+G(u(t))),\qquad
 \mathrm d[M](t)\leq CG(u(t))\,\mathrm dt.
\]
Here \(M\) is a continuous local martingale and \(\mathfrak b\) is the
drift coefficient.  For the exact solution started from a deterministic
state \(x\), there exist constants \(c_0,c_1>0\), depending only on
\(\Phi\), such that whenever
\[
 0<h\leq c_0\lambda^{-1},
 \qquad 0<\theta\leq c_1\eta_*\lambda
\]
one has, for every fixed \(T<\infty\),
\begin{equation}
 \sup_{Nh\leq T}\mathbb E\left[\exp\left\{\theta h\sum_{j=0}^{N-1}G(u(t_j))\right\}\right]
 \leq C_Te^{\eta_*G(x)}.
 \label{eq:gridpoint-exponential-moment}
\end{equation}
On \([t_j,t_{j+1})\), set
\(\omega_h(t)=t_{j+1}-t\).  Integration by parts on one interval gives
\[
 hG(u(t_j))=\int_{t_j}^{t_{j+1}}G(u(s))\,\mathrm ds
 -\int_{t_j}^{t_{j+1}}(t_{j+1}-s)\,\mathrm dG(u(s)).
\]
Summing over \(j=0,\ldots,N-1\) yields
\[
 h\sum_{j=0}^{N-1}G(u(t_j))-\int_0^{Nh}G(u(s))\,\mathrm ds
 =-\int_0^{Nh}\omega_h(s)\mathfrak b(s)\,\mathrm ds
  -\int_0^{Nh}\omega_h(s)\,\mathrm dM(s).
\]
Since \(0\leq\omega_h\leq h\), the drift term satisfies
\[
 \left|\int_0^{Nh}\omega_h(s)\mathfrak b(s)\,\mathrm ds\right|
 \leq h\int_0^{Nh}|\mathfrak b(s)|\,\mathrm ds
 \leq C\lambda h\left(T+\int_0^{Nh}G(u(s))\,\mathrm ds\right).
\]
For \(R(t)=-\int_0^t\omega_h(s)\,\mathrm dM(s)\), the quadratic variation satisfies
\[
 [R](Nh)=\int_0^{Nh}\omega_h(s)^2\,\mathrm d[M](s)
 \leq Ch^2\int_0^{Nh}G(u(s))\,\mathrm ds.
\]
Consequently,
\[
 h\sum_{j=0}^{N-1}G(u(t_j))
 \leq C\lambda hT+(1+C\lambda h)\int_0^{Nh}G(u(s))\,\mathrm ds+R(Nh).
\]
The exponential local martingale associated with \(2\theta R\) is a
positive supermartingale, so
\[
 \mathbb E[\exp\{2\theta R(Nh)-2\theta^2[R](Nh)\}]\leq1.
\]
Writing \(\theta R=(\theta R-\theta^2[R])+\theta^2[R]\) and applying
the Cauchy--Schwarz inequality gives
\[
 \begin{split}
 &\mathbb E\left[\exp\left\{\theta h\sum_{j=0}^{N-1}G(u(t_j))\right\}\right]\\
 &\quad\leq e^{C\theta\lambda hT}
 \left(\mathbb E\left[\exp\left\{
 2\theta(1+C\lambda h)\int_0^{Nh}G(u(s))\,\mathrm ds
 +2\theta^2[R](Nh)\right\}\right]\right)^{1/2}\\
 &\quad\leq e^{C\theta\lambda hT}
 \left(\mathbb E\left[\exp\left\{
 [2\theta(1+C\lambda h)+C\theta^2h^2]
 \int_0^{Nh}G(u(s))\,\mathrm ds\right\}\right]\right)^{1/2}.
 \end{split}
\]

Since \(G\geq0\), the endpoint term can be dropped from the exponent in
\eqref{eq:common-integrated-exp}.  That estimate bounds the right-hand side
provided that
\[
 2\theta(1+C\lambda h)+C\theta^2h^2
 \leq\frac{3\eta_*\lambda}{7}.
\]
Decreasing \(c_0,c_1\) if necessary gives
\eqref{eq:gridpoint-exponential-moment}.

\medskip\noindent\emph{Step 4: high-moment error estimates.}
In \eqref{eq:finite-error-decomposition},
\(\mathbb E_n[\delta_{n+1}^\circ]=0\), so the linear term in a conditional
Taylor expansion vanishes.  For an
\(\mathcal F_{t_n}\)-measurable vector \(z\) and a random vector \(\xi\)
satisfying \(\mathbb E_n[\xi]=0\), that term is
\(2p\|z\|_{L^2}^{2p-2}\langle z,\mathbb E_n[\xi]\rangle=0\).
Bounding the second-order remainder for \(v\mapsto\|v\|_{L^2}^{2p}\) gives
\[
 \mathbb E_n[\|z+\xi\|_{L^2}^{2p}]
 \leq\|z\|_{L^2}^{2p}
 +C_p\|z\|_{L^2}^{2p-2}\mathbb E_n[\|\xi\|_{L^2}^2]
 +C_p\mathbb E_n[\|\xi\|_{L^2}^{2p}].
\]
Apply this estimate with \(z=p_n+\bar\delta_{n+1}\) and
\(\xi=\delta_{n+1}^\circ\).
Young's inequality with parameter \(h\) and
\eqref{eq:finite-local-mean} give the factor
\(h^{4p}h^{-(2p-1)}=h^{2p+1}\) for the mean term.
For the term containing the conditional second moment,
\eqref{eq:finite-local-centered-two} gives
\((h^3)^ph^{-(p-1)}=h^{2p+1}\).
By \eqref{eq:finite-local-centered-moment}, the remaining \(2p\)-moment term
is bounded by \(C_ph^{2p+1}(1+I_6^+(u(t_n)))^{2pq_p}\), since
\(3p\geq2p+1\) for integer \(p>1\) and \(h\leq1\).
Combining these bounds gives
\[
 \mathbb E_n[\|e_{n+1}\|_{L^2}^{2p}]
 \leq(1+C_ph)\|p_n\|_{L^2}^{2p}
      +C_ph^{2p+1}(1+I_6^+(u(t_n)))^{2pq_p}.
\]
Substituting \eqref{eq:finite-propagated-stability} and using
\(e^{-2p\lambda h}\leq1\), we obtain
\begin{equation}
 \mathbb E_n[\|e_{n+1}\|_{L^2}^{2p}]
 \leq(1+C_ph)e^{pC_{\mathrm K}h(1+G(u(t_n)))}\|e_n\|_{L^2}^{2p}
 +C_ph^{2p+1}(1+I_6^+(u(t_n)))^{2pq_p}.
 \label{eq:finite-error-recurrence}
\end{equation}
Define
\[
 \Gamma_0=1,
 \qquad
 \Gamma_n=\prod_{j=0}^{n-1}
 \left[(1+C_ph)e^{pC_{\mathrm K}h(1+G(u(t_j)))}\right]^{-1}.
\]
Since \(\Gamma_{n+1}\) depends only on states at times \(t_j\leq t_n\),
it is \(\mathcal F_{t_n}\)-measurable.  The identity
\[
 \Gamma_{n+1}(1+C_ph)e^{pC_{\mathrm K}h(1+G(u(t_n)))}=\Gamma_n
\]
therefore gives, after multiplying \eqref{eq:finite-error-recurrence} by
\(\Gamma_{n+1}\),
\[
 \begin{split}
 \mathbb E_n[\Gamma_{n+1}\|e_{n+1}\|_{L^2}^{2p}]
 &\leq\Gamma_n\|e_n\|_{L^2}^{2p}\\
 &\quad+C_ph^{2p+1}\Gamma_{n+1}
 (1+I_6^+(u(t_n)))^{2pq_p}.
 \end{split}
\]
We take expectations and sum over \(j=0,\ldots,n-1\).  Since
\(e_0=0\), \(0<\Gamma_{n+1}\leq1\), and \(nh\leq T\), the polynomial
moment bound \eqref{eq:finite-local-source-moment} yields
\begin{equation}
 \begin{split}
 \mathbb E[\Gamma_n\|e_n\|_{L^2}^{2p}]
 &\leq C_ph^{2p+1}\sum_{j=0}^{n-1}
 \mathbb E\!\left[\Gamma_{j+1}(1+I_6^+(u(t_j)))^{2pq_p}\right]\\
 &\leq C_Th^{2p}.
 \end{split}
 \label{eq:weighted-error-product}
\end{equation}
To recover the unweighted second moment from this weighted \(2p\)-moment
estimate, apply H\"older's inequality to obtain
\[
 \mathbb E[\|e_n\|_{L^2}^2]
 \leq\{\mathbb E[\Gamma_n\|e_n\|_{L^2}^{2p}]\}^{1/p}
 \{\mathbb E[\Gamma_n^{-1/(p-1)}]\}^{(p-1)/p}.
\]
For the second factor, observe that
\[
 \Gamma_n^{-1/(p-1)}
 =(1+C_ph)^{n/(p-1)}
 \exp\left\{\frac{pC_{\mathrm K}}{p-1}nh\right\}
 \exp\left\{\frac{pC_{\mathrm K}}{p-1}
 h\sum_{j=0}^{n-1}G(u(t_j))\right\}.
\]
The first two factors are bounded by \(C_T\) for \(nh\leq T\).
Since \(p/(p-1)\leq2\), choose \(\lambda_{\mathrm{fin}}\geq1\) so that
\(2C_{\mathrm K}\leq c_1\eta_*\lambda_{\mathrm{fin}}\), and take
\(h_0\leq\min\{1,c_0/\lambda\}\).
For \(\lambda\geq\lambda_{\mathrm{fin}}\),
\eqref{eq:gridpoint-exponential-moment} then gives
\(\mathbb E[\Gamma_n^{-1/(p-1)}]\leq C_T\).
Combining this bound with \eqref{eq:weighted-error-product} in
H\"older's inequality proves \eqref{eq:finite-strong-order-one}.
\end{proof}

\begin{remark}[Random initial data in Theorem~\ref{thm:finite-strong-order-one}]
\label{rem:finite-strong-random-data}
The same conclusion holds when \(u(0)=U^0=u_0\) almost surely, where
\(u_0\) takes values in \(H^6_0\), is \(\mathcal F_0\)-measurable, and satisfies
\[
 \mathbb E[(I_6^+(u_0))^p]<\infty\quad\hbox{for every finite }p,
 \qquad \mathbb E[e^{\eta_*G(u_0)}]<\infty.
\]
Only finitely many of these polynomial moments enter the proof.
\end{remark}

%
\subsection{Strong convergence uniformly in time}
\label{sec:uniform-strong}

We now establish a uniform-in-time mean-square error estimate with an
exponential weight.  We estimate the conditional mean and the zero-mean
part of each local error separately.  The exponential weight is evaluated
at \(U^n\), which also determines the coefficient in the one-sided continuous
dependence estimate.

Throughout this subsection, assume
\(\Phi\in L_2(\mathfrak U,H^6_0)\) and \(\lambda\geq1\).
Let \(C_0,C_{\mathrm{qv}}>0\) be as in
\eqref{eq:G-semimartingale-bounds}, and choose
\begin{equation}
 0<a\leq\min\left\{
 \frac{\eta_*}{8},\frac{1}{7C_{\mathrm{qv}}},\frac{1}{4C_0},1\right\}.
 \label{eq:common-long-time-weight}
\end{equation}
Along the exact solution and each Ornstein--Uhlenbeck step,
\eqref{eq:G-semimartingale-bounds} gives
\begin{equation}
 \begin{gathered}
 \mathrm dG=\mathfrak b\,\mathrm dt+\mathrm dM,
 \qquad
 \mathfrak b\leq-\frac{5\lambda}{7}G+C_0\lambda,
 \qquad |\mathfrak b|\leq C\lambda(1+G),\\
 \mathrm d[M](t)=q_t\,\mathrm dt,
 \qquad 0\leq q_t\leq C_{\mathrm{qv}}G.
 \end{gathered}
 \label{eq:G-semimartingale-uniform-proof}
\end{equation}

To retain the damping at each step, we need a more precise estimate than
\eqref{eq:common-endpoint-exp}.  The following lemma gives a bound with a
factor \(1+O(h)\) for fixed \(\lambda\) over a single Ornstein--Uhlenbeck step.

\begin{lemma}[Exponential weight estimates over an Ornstein--Uhlenbeck step]
\label{lem:moving-ou-weight}
Let \(v\) be an \(H^6_0\)-valued, \(\mathcal F_t\)-measurable random
state, and set
\[
 Y_s=\mathcal O^\omega_{t+s,t}(v),
 \qquad 0\leq s\leq h.
\]
Define
\begin{equation}
 \theta_s=2ae^{-3\lambda(h-s)/7},
 \qquad R_s=e^{\theta_sG(Y_s)}.
 \label{eq:moving-weight-definitions}
\end{equation}
In particular,
\[
 R_0=e^{2ae^{-3\lambda h/7}G(v)},
 \qquad R_h=e^{2aG(Y_h)}.
\]
For \(m=1,2\),
\begin{equation}
 \mathbb E\left[R_h^m\mid\mathcal F_t\right]
 \leq e^{2maC_0\lambda h}R_0^m.
 \label{eq:moving-endpoint-sharp}
\end{equation}
There is a polynomial \(P\) with nonnegative coefficients,
independent of \(t,h,v\), such that, for every fixed \(\lambda\geq1\) and
\(0<h\leq\lambda^{-1}\),
\begin{equation}
 \mathbb E\left[
 |R_h-R_0|^2\mid\mathcal F_t\right]
 \leq C_{a,\lambda}hR_0^2P(I_6^+(v)).
 \label{eq:moving-weight-increment}
\end{equation}
\end{lemma}

\begin{proof}
We first take deterministic \(v\).  Applying It\^o's formula to
\(R_s^m=e^{m\theta_sG(Y_s)}\), with
\(\frac{\mathrm d\theta_s}{\mathrm ds}=3\lambda\theta_s/7\), gives
\[
 \mathrm dR_s^m
 =R_s^m\left\{
 \left(m\frac{\mathrm d\theta_s}{\mathrm ds}G(Y_s)+m\theta_s\mathfrak b(s)
       +\frac{m^2\theta_s^2}{2}q_s\right)\mathrm ds
       +m\theta_s\,\mathrm dM_s\right\}.
\]
The coefficient multiplying \(G(Y_s)\) is bounded above by
\[
 m\theta_s\left\{-\frac{2\lambda}{7}
                   +\frac{m\theta_sC_{\mathrm{qv}}}{2}\right\}.
\]
For \(m=1,2\), we have \(m\theta_s\leq4a\), and
\eqref{eq:common-long-time-weight} makes this coefficient nonpositive.
Thus the drift coefficient of \(R_s^m\) is bounded above by
\(mC_0\lambda\theta_sR_s^m\), so
\[
 \exp\left\{-mC_0\lambda\int_0^s\theta_r\,\mathrm dr\right\}R_s^m
\]
is a nonnegative local supermartingale and hence a supermartingale.
Taking expectations and using
\(\int_0^h\theta_r\,\mathrm dr\leq2ah\) proves
\eqref{eq:moving-endpoint-sharp}.

For the increment estimate, taking \(m=1\) in the It\^o formula and using
\eqref{eq:G-semimartingale-uniform-proof} gives
\[
 \mathrm dR_s=R_s\{A_s\,\mathrm ds+\theta_s\,\mathrm dM_s\},
 \qquad |A_s|\leq C_a\lambda(1+G(Y_s)).
\]
Integrating from \(0\) to \(h\), we obtain
\[
 R_h-R_0=\int_0^hR_sA_s\,\mathrm ds+\int_0^hR_s\theta_s\,\mathrm dM_s.
\]
We estimate these two integrals using bounds for
\(\mathbb E[R_s^2(1+G(Y_s))^j]\), \(j=1,2\).
The polynomial and endpoint exponential estimates of
Proposition~\ref{prop:uniform-lyapunov} also apply to the Ornstein--Uhlenbeck
flow started from \(v\).  Indeed, for any fixed \(s>0\), this flow has the
same endpoint as a splitting step of length \(s\) started from
\(\mathcal K_{-s}v\), whose modified invariants equal those of \(v\).
Since \(4\theta_s\leq8a\leq\eta_*\), these estimates, coercivity, and
H\"older's inequality give, for \(j=1,2\),
\[
 \begin{split}
 \mathbb E[R_s^2(1+G(Y_s))^j]
 &\leq
 \bigl(\mathbb E[(1+G(Y_s))^{2j}]\bigr)^{1/2}
 \bigl(\mathbb E[e^{4\theta_sG(Y_s)}]\bigr)^{1/2}\\
 &\leq Ce^{2\theta_se^{-3\lambda s/7}G(v)}
       P(I_6^+(v)).
 \end{split}
\]
Since \(\theta_se^{-3\lambda s/7}=2ae^{-3\lambda h/7}\), the exponential
factor equals \(R_0^2\).  Thus
\begin{equation}
 \mathbb E[R_s^2(1+G(Y_s))^j]
 \leq CR_0^2P(I_6^+(v)),
 \qquad j=1,2.
 \label{eq:moving-weight-mixed}
\end{equation}
Estimate \eqref{eq:moving-weight-mixed}, together with the bounds on
\(A_s\) and \(q_s\), gives the integrability required for Cauchy--Schwarz
and It\^o's isometry.  Hence
\[
 \begin{split}
 \mathbb E[|R_h-R_0|^2]
 &\leq2h\int_0^h\mathbb E[R_s^2|A_s|^2]\,\mathrm ds
       +2\int_0^h\mathbb E[R_s^2\theta_s^2q_s]\,\mathrm ds\\
 &\leq C_a(\lambda^2h^2+h)
       R_0^2P(I_6^+(v)).
 \end{split}
\]
Since \(\lambda^2h^2+h\leq(\lambda+1)h\) for \(h\leq\lambda^{-1}\),
this proves \eqref{eq:moving-weight-increment}.  For an
\(\mathcal F_t\)-measurable initial state \(v\), independence of the future
Wiener increments gives the stated conditional estimates.
\end{proof}

To control the growth in the continuous dependence estimate, set
\begin{equation}
 \lambda_{\mathrm{uni}}=\lambda_{\mathrm{uni}}(a,\Phi)
 =\max\left\{
 1,\frac{C_{\mathrm K}e^{3/7}}{2a(3/7)},2C_{\mathrm K}\right\}.
 \label{eq:common-damping-threshold}
\end{equation}
For each fixed \(\lambda\geq\lambda_{\mathrm{uni}}\), choose
\begin{equation}
 0<h\leq h_*\leq\min\{\lambda^{-1},1\},
 \label{eq:common-step-threshold}
\end{equation}
where \(h_*=h_*(a,\lambda,\Phi)>0\) is also no larger than the step
threshold in the local error estimates.

\begin{theorem}[Uniform-in-time strong convergence]
\label{thm:uniform-strong}
Assume that \(u_0\in H^6_0\) is deterministic and
\(\Phi\in L_2(\mathfrak U,H^6_0)\).  Choose \(a\) as in
\eqref{eq:common-long-time-weight}, fix \(\lambda\geq\lambda_{\mathrm{uni}}\), and
take \(h\) as in \eqref{eq:common-step-threshold}.  Let
\(u(0)=U^0=u_0\), and drive the exact solution and the splitting
approximation by the same Wiener process.  Then
\begin{equation}
 \sup_{n\geq0}\mathbb E\left[
 e^{2aG(U^n)}\|u(nh)-U^n\|_{L^2}^2\right]\leq Ch^2.
 \label{eq:uniform-weighted-strong-one}
\end{equation}
Consequently,
\begin{equation}
 \sup_{n\geq0}\mathbb E[\|u(nh)-U^n\|_{L^2}^2]\leq Ch^2,
 \qquad 0<h\leq h_*.
 \label{eq:uniform-strong}
\end{equation}
Hence the splitting approximation has strong order \(1\) uniformly in time.
The constant may depend on \(u_0,a,\lambda,\Phi\) and the fixed Lyapunov
constants, but is independent of \(n\), the final time, and
\(h\in(0,h_*]\).
\end{theorem}

\begin{proof}
Write \(\mathbb E_n[\,\cdot\,]=
\mathbb E[\,\cdot\mid\mathcal F_{t_n}]\), and set
\[
 x=u(t_n),\qquad y=U^n,\qquad e_n=x-y.
\]

We use \(P\) to denote polynomials in the displayed arguments with
nonnegative coefficients, independent of \(n\) and \(h\), enlarged when
necessary.

\medskip\noindent\emph{Step 1: the weighted error and the local error.}
The difference between the two splitting updates
\[
 p_n=\mathcal L^\omega_{t_{n+1},t_n}x
      -\mathcal L^\omega_{t_{n+1},t_n}y
 =e^{-\lambda h}(\mathcal K_hx-\mathcal K_hy)
\]
is \(\mathcal F_{t_n}\)-measurable, because the common stochastic
convolution cancels.  Decompose the local error as
\[
 \begin{gathered}
 \delta_{n+1}=\mathcal S^\omega_{t_{n+1},t_n}x
               -\mathcal L^\omega_{t_{n+1},t_n}x
 =\bar\delta_{n+1}+\delta_{n+1}^\circ,\\
 \bar\delta_{n+1}=\mathbb E_n[\delta_{n+1}],
 \qquad \mathbb E_n[\delta_{n+1}^\circ]=0.
 \end{gathered}
\]
Proposition~\ref{prop:local-error-moments} gives
\begin{equation}
 \begin{gathered}
 \|\bar\delta_{n+1}\|_{L^2}\leq Ch^2P(I_6^+(x)),
 \qquad
 \mathbb E_n[\|\delta_{n+1}^\circ\|_{L^2}^2]\leq Ch^3P(I_6^+(x)),\\
 \mathbb E_n[\|\delta_{n+1}\|_{L^2}^4]\leq Ch^6P(I_6^+(x)).
 \end{gathered}
 \label{eq:local-input-uniform-one}
\end{equation}
Define
\[
 Z_n=e^{2aG(U^n)}\|e_n\|_{L^2}^2,
 \qquad W_{n+1}=e^{2aG(U^{n+1})},
 \qquad z_n=e^{2ae^{-3\lambda h/7}G(y)}.
\]
Apply Lemma~\ref{lem:moving-ou-weight} with \(t=t_n\) and
\(v=\mathcal K_hy\).  By \eqref{eq:modified-invariance},
\(G(v)=G(y)\) and \(I_6^+(v)=I_6^+(y)\), so
\[
 Y_h=U^{n+1},\qquad R_h=W_{n+1},\qquad R_0=z_n.
\]
Since \(e_{n+1}=p_n+\delta_{n+1}\),
\begin{equation}
 \begin{split}
 \mathbb E_n[Z_{n+1}]
 ={}&(\mathbb E_n[W_{n+1}])\|p_n\|_{L^2}^2
 +2\langle p_n,\mathbb E_n[W_{n+1}\delta_{n+1}]\rangle\\
 &+\mathbb E_n[W_{n+1}\|\delta_{n+1}\|_{L^2}^2].
 \end{split}
 \label{eq:weighted-error-exact-expansion}
\end{equation}

\medskip\noindent\emph{Step 2: the weighted local terms.}
The endpoint weight and the zero-mean part of the local error depend on
the Wiener increments over \([t_n,t_{n+1}]\), so
\(\mathbb E_n[W_{n+1}\delta_{n+1}^\circ]\) need not vanish.
Since \(z_n\) is \(\mathcal F_{t_n}\)-measurable, we have
\[
 \mathbb E_n[W_{n+1}\delta_{n+1}]
 =\bar\delta_{n+1}\mathbb E_n[W_{n+1}]
  +\mathbb E_n[(W_{n+1}-z_n)\delta_{n+1}^\circ].
\]
For the second term, conditional Cauchy--Schwarz,
\eqref{eq:moving-weight-increment}, and
\eqref{eq:local-input-uniform-one} give
\[
 \begin{split}
 \|\mathbb E_n[(W_{n+1}-z_n)\delta_{n+1}^\circ]\|_{L^2}
 &\leq
 (\mathbb E_n[|W_{n+1}-z_n|^2])^{1/2}
 (\mathbb E_n[\|\delta_{n+1}^\circ\|_{L^2}^2])^{1/2}\\
 &\leq Ch^2z_nP(I_6^+(x),I_6^+(y)).
\end{split}
\]
Together with \eqref{eq:moving-endpoint-sharp} and
\eqref{eq:local-input-uniform-one}, this proves
\begin{equation}
 \|\mathbb E_n[W_{n+1}\delta_{n+1}]\|_{L^2}
 \leq Ch^2z_nP(I_6^+(x),I_6^+(y)).
 \label{eq:weighted-centered-cross}
\end{equation}
The case \(m=2\) of \eqref{eq:moving-endpoint-sharp} also yields
\begin{equation}
 \begin{split}
 \mathbb E_n[W_{n+1}\|\delta_{n+1}\|_{L^2}^2]
 &\leq(\mathbb E_n[W_{n+1}^2])^{1/2}
       (\mathbb E_n[\|\delta_{n+1}\|_{L^2}^4])^{1/2}\\
 &\leq Ch^3z_nP(I_6^+(x)).
 \end{split}
 \label{eq:weighted-local-square-new}
\end{equation}

\medskip\noindent\emph{Step 3: KdV growth against Ornstein--Uhlenbeck
damping.}
Here Lemma~\ref{lem:one-sided-kdv-stability} is applied with the coefficient
evaluated at \(y=U^n\), so that it matches the state in the exponential
weight.  It gives
\[
 z_n\|p_n\|_{L^2}^2
 \leq\exp\left\{-2\lambda h+C_{\mathrm K}h
 +\left[C_{\mathrm K}h-2a(1-e^{-3\lambda h/7})\right]G(y)\right\}Z_n.
\]
For \(h\leq\lambda^{-1}\),
\[
 1-e^{-3\lambda h/7}\geq\frac{3\lambda h}{7}e^{-3/7}.
\]
The choice of \(\lambda_{\mathrm{uni}}\) makes the coefficient of \(G(y)\)
nonpositive, so
\begin{equation}
 z_n\|p_n\|_{L^2}^2
 \leq e^{-(2\lambda-C_{\mathrm K})h}Z_n.
 \label{eq:z-propagation-bound}
\end{equation}
In the cross term from \eqref{eq:weighted-error-exact-expansion}, apply
Young's inequality to \(\sqrt{z_n}\|p_n\|_{L^2}\) and
\(Ch^2\sqrt{z_n}P\).  For every \(\varepsilon>0\),
\[
 Ch^2z_n\|p_n\|_{L^2}P
 \leq\varepsilon h z_n\|p_n\|_{L^2}^2
       +C_\varepsilon h^3z_nP^2.
\]
Combining this bound with \eqref{eq:moving-endpoint-sharp},
\eqref{eq:weighted-local-square-new}, and
\eqref{eq:z-propagation-bound}, and using \(z_n\leq e^{2aG(y)}\), gives
\[
 \mathbb E_n[Z_{n+1}]
 \leq\{e^{2aC_0\lambda h}+\varepsilon h\}
       e^{-(2\lambda-C_{\mathrm K})h}Z_n
       +Ch^3\mathcal V(x,y),
\]
where, after all polynomial enlargements and the square in Young's
inequality,
\[
 \mathcal V(x,y)=e^{2aG(y)}P(I_6^+(x),I_6^+(y)).
\]
Set
\begin{equation}
 \kappa_0=(2-2aC_0)\lambda-C_{\mathrm K}\geq\lambda>0.
 \label{eq:uniform-contraction-rate}
\end{equation}
The last inequality follows from \(aC_0\leq1/4\) and
\(C_{\mathrm K}\leq\lambda/2\).  Since
\[
 \begin{split}
 \{e^{2aC_0\lambda h}+\varepsilon h\}
 e^{-(2\lambda-C_{\mathrm K})h}
 &=e^{-\kappa_0h}
   (1+\varepsilon h e^{-2aC_0\lambda h})\\
 &\leq e^{-(\kappa_0-\varepsilon)h},
 \end{split}
\]
the choice \(\varepsilon=\kappa_0/2\) yields
\begin{equation}
 \mathbb E_n[Z_{n+1}]
 \leq e^{-\kappa_0h/2}Z_n+Ch^3\mathcal V(x,y).
 \label{eq:weighted-contracting-recurrence}
\end{equation}
\medskip\noindent\emph{Step 4: uniform moment bounds and summation.}
Fix an integer \(d\geq1\) bounding the degree in each variable of the
final polynomial \(P\) in \(\mathcal V\).  Since \(I_6^+\geq1\),
\(P(I_6^+(x),I_6^+(y))\leq C\bigl(I_6^+(x)\bigr)^d\bigl(I_6^+(y)\bigr)^d\).
H\"older's inequality with exponents \(2,4,4\) gives
\[
 \begin{split}
 \mathbb E[\mathcal V(u(t_n),U^n)]
 &\leq C\bigl(\mathbb E\left[\bigl(I_6^+(u(t_n))\bigr)^{2d}\right]\bigr)^{1/2}
       \bigl(\mathbb E\left[\bigl(I_6^+(U^n)\bigr)^{4d}\right]\bigr)^{1/4}\\
 &\qquad\times\bigl(\mathbb E[e^{8aG(U^n)}]\bigr)^{1/4}.
 \end{split}
\]
Proposition~\ref{prop:uniform-lyapunov} bounds all three factors uniformly in
the time index and step size.  This specifies the required polynomial moment
orders, while the largest exponential parameter is \(8a\leq\eta_*\).  Hence
\begin{equation}
 \sup_{0<h\leq h_*}\sup_{n\geq0}
 \mathbb E[\mathcal V(u(t_n),U^n)]<\infty.
 \label{eq:uniform-recurrence-source}
\end{equation}
As \(Z_0=0\), expectation and iteration in
\eqref{eq:weighted-contracting-recurrence} now give
\[
 \sup_{n\geq0}\mathbb E[Z_n]
 \leq\frac{Ch^3}{1-e^{-\kappa_0h/2}}
 \leq Ch^2.
\]
This proves \eqref{eq:uniform-weighted-strong-one}.  Since the weight is
at least one, \eqref{eq:uniform-strong} follows.
\end{proof}

The mean-square error estimate with an exponential weight also controls
the distance-like function used in Section~\ref{sec:weak-invariant}.
Cauchy--Schwarz and the exact endpoint exponential estimate give
\begin{equation}
 \begin{split}
 \sup_{n\geq0}\mathbb E[d_a(u(nh),U^n)]
 &\leq\sup_{n\geq0}
 \left(\mathbb E\left[e^{2aG(U^n)}
       \|u(nh)-U^n\|_{L^2}^2\right]\right)^{1/2}
 \left(\mathbb E[e^{2aG(u(nh))}]\right)^{1/2}\\
 &\leq Ch.
 \end{split}
 \label{eq:uniform-weighted-cost-one}
\end{equation}

\begin{remark}[Random initial data in Theorem~\ref{thm:uniform-strong}]
\label{rem:uniform-strong-random-data}
The same estimates hold when \(u(0)=U^0=u_0\) almost surely, where
\(u_0\) is \(H^6_0\)-valued and \(\mathcal F_0\)-measurable, and
satisfies
\[
 \mathbb E\left[\bigl(I_6^+(u_0)\bigr)^{4d}\right]<\infty,
 \qquad \mathbb E[e^{8aG(u_0)}]<\infty.
\]
Here \(d\geq1\) bounds the degree in each variable of the polynomial \(P\)
in Step~4.  Future Wiener increments are independent of \(\mathcal F_0\).
H\"older's inequality and the conditional estimates of
Proposition~\ref{prop:uniform-lyapunov} yield
\eqref{eq:uniform-recurrence-source}.  The constants depend on the initial
law only through upper bounds for these two moments.
\end{remark}

\section{Weak approximation and invariant measures}
\label{sec:weak-invariant}

In this section, we use \eqref{eq:uniform-weighted-cost-one} to estimate
the weak error for observables that are Lipschitz with respect to \(d_a\).
We then combine this estimate with contraction of the exact Markov
semigroup to compare invariant measures.

Throughout this section, assume \(\Phi\in L_2(\mathfrak U,H^6_0)\).
We work on \(H^2_0\), choose \(a\) as in
\eqref{eq:common-long-time-weight}, and take
\(\lambda\geq\lambda_{\mathrm{uni}}\) and \(0<h\leq h_*\), with
\(h_*\) as in \eqref{eq:common-step-threshold}.

\subsection{Weighted observables}

For \(\varphi:H^2_0\to\mathbb R\), define
\begin{equation}
 [\varphi]_{d_a}=
 \sup_{x\neq y}\frac{|\varphi(x)-\varphi(y)|}{d_a(x,y)}.
 \label{eq:weighted-lipschitz-seminorm}
\end{equation}
Let \([\varphi]_{d_a}<\infty\).  Subtracting \(\varphi(0)\) does not change
this seminorm, so we may assume \(\varphi(0)=0\).  By
\eqref{eq:I2-coercive},
\[
 |\varphi(x)|\leq[\varphi]_{d_a}d_a(x,0)
 \leq C_a[\varphi]_{d_a}e^{2aG(x)}.
\]
Let \(\nu\) be a probability measure with \(\nu(H^6_0)=1\) and
\(\int e^{2aG(x)}\,\nu(\mathrm dx)<\infty\).
The endpoint exponential estimate in Proposition~\ref{prop:uniform-lyapunov}
gives
\begin{equation}
 \begin{split}
 &\sup_{t\geq0}\int e^{2aG(x)}\,(\nu P_t)(\mathrm dx)
 +\sup_{0<h\leq h_*}\sup_{n\geq0}
       \int e^{2aG(x)}\,(\nu Q_h^n)(\mathrm dx)\\
 &\qquad\leq2C_E\int e^{2aG(x)}\,\nu(\mathrm dx)<\infty.
 \end{split}
 \label{eq:weak-observable-integrability}
\end{equation}
Thus \(\varphi\) is integrable under \(\nu P_t\) and \(\nu Q_h^n\)
for every \(t\geq0\), \(n\geq0\), and \(0<h\leq h_*\).

\subsection{Weak error bounds uniform in time}

Estimate \eqref{eq:uniform-weighted-cost-one} gives a first-order weak
error bound that is uniform in time.
For probability measures with finite second \(L^2\)-moments, regarded as
measures on \(L^2_0\), write
\begin{equation}
 \mathcal W_2^{L^2}(\nu_1,\nu_2)
 =\inf_{\pi\in\Gamma(\nu_1,\nu_2)}
   \left(\int\|x-y\|_{L^2}^2\,\pi(\mathrm dx,\mathrm dy)\right)^{1/2},
 \label{eq:L2-Wasserstein-two}
\end{equation}
where \(\Gamma(\nu_1,\nu_2)\) is the set of their couplings.

\begin{theorem}[Uniform-in-time weighted weak error]
\label{thm:uniform-weak}
Assume \(\Phi\in L_2(\mathfrak U,H^6_0)\), choose \(a\) as in
\eqref{eq:common-long-time-weight}, and fix
\(\lambda\geq\lambda_{\mathrm{uni}}\).
Let \(h_*\) be as in \eqref{eq:common-step-threshold}.
As in Step~4 of the proof of Theorem~\ref{thm:uniform-strong}, choose an
integer \(d\geq1\) such that the polynomial \(P\) has degree at most \(d\)
in each variable.
Let \((\nu_h)_{0<h\leq h_*}\) be probability measures on \(H^6_0\)
satisfying
\begin{equation}
 \sup_{0<h\leq h_*}
 \int_{H^6_0}\{\bigl(I_6^+(x)\bigr)^{4d}+e^{8aG(x)}\}\,\nu_h(\mathrm dx)<\infty.
 \label{eq:weak-initial-order-one}
\end{equation}
Then, for every observable \(\varphi\) with \([\varphi]_{d_a}<\infty\),
\begin{equation}
 \sup_{n\geq0}|\nu_hQ_h^n\varphi-\nu_hP_{nh}\varphi|
 \leq C[\varphi]_{d_a}h,
 \qquad 0<h\leq h_*.
 \label{eq:uniform-weak}
\end{equation}
The same laws also satisfy
\begin{equation}
 \sup_{n\geq0}\mathcal W_2^{L^2}(\nu_hQ_h^n,\nu_hP_{nh})
 \leq Ch.
 \label{eq:uniform-transient-W2}
\end{equation}
For fixed \(a,\lambda,\Phi\), the constants may depend on the uniform
moment bound in \eqref{eq:weak-initial-order-one}, but are independent of
\(n\), the final time, and \(h\in(0,h_*]\).
\end{theorem}

\begin{proof}
Take a common initial state of law \(\nu_h\), independently of the future
Wiener increments, and drive the exact solution and the splitting
approximation by the same Wiener process.  Under
\eqref{eq:weak-initial-order-one}, Remark~\ref{rem:uniform-strong-random-data}
gives
\begin{equation}
 \sup_{n\geq0}
 \mathbb E\left[e^{2aG(U^n)}\|u(nh)-U^n\|_{L^2}^2\right]
 \leq Ch^2,
 \qquad 0<h\leq h_*.
 \label{eq:random-initial-weighted-strong}
\end{equation}
with a constant independent of \(h\).
By the definition of \(d_a\) and the Cauchy--Schwarz inequality,
\[
 \begin{aligned}
 \mathbb E[d_a(u(nh),U^n)]
 &=\mathbb E\left[e^{aG(u(nh))}e^{aG(U^n)}
       \|u(nh)-U^n\|_{L^2}\right]\\
 &\leq\left(\mathbb E[e^{2aG(u(nh))}]\right)^{1/2}
       \left(\mathbb E\left[e^{2aG(U^n)}
       \|u(nh)-U^n\|_{L^2}^2\right]\right)^{1/2}.
 \end{aligned}
\]
The first factor is uniformly bounded by
\eqref{eq:weak-observable-integrability} and
\eqref{eq:weak-initial-order-one}, while
\eqref{eq:random-initial-weighted-strong} bounds the second by \(Ch\).
Hence
\begin{equation}
 \sup_{n\geq0}\mathbb E[d_a(u(nh),U^n)]\leq Ch.
 \label{eq:random-initial-weighted-cost}
\end{equation}
Condition~\eqref{eq:weak-initial-order-one} and
\eqref{eq:weak-observable-integrability} ensure that \(\varphi\) is integrable
under both marginal laws.  Integrating the defining Lipschitz bound against
this coupling proves \eqref{eq:uniform-weak}.
Since the exponential weight is at least one,
\eqref{eq:random-initial-weighted-strong} also bounds the mean-square error.
The same coupling then gives \eqref{eq:uniform-transient-W2}.
\end{proof}

\subsection{Approximation of invariant measures}

We apply the uniform weak estimate with the invariant measure of the
splitting approximation as the initial law.  The comparison with the exact invariant measure also
uses contraction of the exact Markov semigroup.

For probability measures \(\nu_1,\nu_2\) satisfying
\[
 \int_{H^2_0}d_a(x,0)\,\nu_i(\mathrm dx)<\infty,
 \qquad i=1,2,
\]
define the weighted dual distance
\begin{equation}
 \mathcal D_{d_a}(\nu_1,\nu_2)
 =\sup_{\substack{[\varphi]_{d_a}\leq1\\\varphi(0)=0}}
 |\nu_1\varphi-\nu_2\varphi|.
 \label{eq:dual-distance}
\end{equation}
The triangle inequality for \(\mathcal D_{d_a}\) follows directly from
\eqref{eq:dual-distance}.

We recall from \cite[Theorem~7.1, Lemma~7.3, and (7.32)]{GHMR2024} the
contraction and exponential moment estimates for the exact dynamics,
together with the existence, uniqueness, and exponential integrability of
its invariant measure.
The proof of Theorem~7.1 gives an exponent
\(a_{\mathrm{con}}=a_{\mathrm{con}}(\Phi)>0\), which we fix small enough
that the recalled invariant-measure moment bound is valid for every
exponent in \([0,2a_{\mathrm{con}}]\).
For every \(0<a\leq a_{\mathrm{con}}\), there exists
\(\lambda_{\mathrm w}(a,\Phi)\geq1\) such that, for all
\(\lambda\geq\lambda_{\mathrm w}(a,\Phi)\), \(x,y\in H^2_0\), and \(t\geq0\),
the transition probabilities \(P_t(x,\cdot)\) and \(P_t(y,\cdot)\)
admit a coupling \(\pi^{x,y}_t\) with
\[
 \int d_a(\xi,\zeta)\,\pi^{x,y}_t(\mathrm d\xi,\mathrm d\zeta)
 \leq2e^{-\lambda t/2}d_a(x,y).
\]
Throughout this subsection, assume in addition \(a\leq a_{\mathrm{con}}\)
and \(\lambda\geq\lambda_{\mathrm w}(a,\Phi)\).
For \([\varphi]_{d_a}<\infty\), the recalled solution moment bound gives
\(P_t|\varphi|(x)<\infty\) for \(x\in H^2_0\) and \(t\geq0\).
Integrating the bound defining \([\varphi]_{d_a}\) against this coupling gives
\[
 |P_t\varphi(x)-P_t\varphi(y)|
 \leq[\varphi]_{d_a}\int d_a(\xi,\zeta)\,\pi^{x,y}_t(\mathrm d\xi,\mathrm d\zeta).
\]
Using the recalled contraction estimate, dividing by \(d_a(x,y)\), and
taking the supremum over \(x\ne y\) yields
\begin{equation}
 [P_t\varphi]_{d_a}
 \leq2e^{-\lambda t/2}[\varphi]_{d_a},
 \qquad t\geq0.
 \label{eq:observable-contraction}
\end{equation}
For probability measures \(\nu_1,\nu_2\) as above and \(t\geq0\) such that
\[
 \int_{H^2_0}d_a(x,0)\,(\nu_iP_t)(\mathrm dx)<\infty,
 \qquad i=1,2,
\]
set \(\psi=P_t\varphi-P_t\varphi(0)\).  Then
\(\psi(0)=0\), and \eqref{eq:observable-contraction} gives
\([\psi]_{d_a}\leq2e^{-\lambda t/2}[\varphi]_{d_a}\).  For every \(\varphi\)
admissible in \eqref{eq:dual-distance},
\[
 |(\nu_1P_t-\nu_2P_t)\varphi|
 =|\nu_1\psi-\nu_2\psi|
 \leq[\psi]_{d_a}\mathcal D_{d_a}(\nu_1,\nu_2).
\]
Taking the supremum over such \(\varphi\) gives
\begin{equation}
 \mathcal D_{d_a}(\nu_1P_t,\nu_2P_t)
 \leq2e^{-\lambda t/2}\mathcal D_{d_a}(\nu_1,\nu_2).
 \label{eq:dual-distance-contraction}
\end{equation}

Under the standing assumptions of this subsection, the exact dynamics has
a unique invariant probability measure \(\mu\).  Since
\(a\leq a_{\mathrm{con}}\), the recalled invariant-measure moment bound
applies at exponent \(2a\) and gives
\[
 \int_{H^2_0}e^{2aG(x)}\,\mu(\mathrm dx)<\infty.
\]
By \eqref{eq:weighted-cost-a}, \eqref{eq:I2-coercive} and
\eqref{eq:G-definition}, \(\|x\|_{L^2}^2\leq CG(x)\) and
\(d_a(x,0)\leq C_a e^{2aG(x)}\); hence
\begin{equation}
 \int_{H^2_0}
 \bigl(d_a(x,0)+\|x\|_{L^2}^2\bigr)\,\mu(\mathrm dx)<\infty.
 \label{eq:exact-invariant-integrability}
\end{equation}

\begin{theorem}[Invariant measure approximation]
\label{thm:invariant-laws}
Assume \(\Phi\in L_2(\mathfrak U,H^6_0)\), choose \(a\) as in
\eqref{eq:common-long-time-weight} with \(a\leq a_{\mathrm{con}}\), and fix
\(\lambda\geq\max\{\lambda_{\mathrm{uni}},\lambda_{\mathrm w}(a,\Phi)\}\).
Let \(h_*\) be as in \eqref{eq:common-step-threshold}.
For every \(0<h\leq h_*\), the Feller kernel \(Q_h\) has a unique invariant
probability measure \(\mu_h\) on \(H^2_0\).  These measures satisfy
\begin{equation}
 \mu_h(H^6_0)=1,
 \qquad
 \sup_{0<h\leq h_*}\int_{H^6_0}
       \{\bigl(I_6^+(x)\bigr)^{4d}+e^{8aG(x)}\}\,\mu_h(\mathrm dx)<\infty.
 \label{eq:invariant-mixed-bound}
\end{equation}
Let \(\mu\) denote the invariant probability measure of the exact dynamics
introduced above.  Then
\begin{align}
 \mathcal D_{d_a}(\mu_h,\mu)&\leq Ch,
 \label{eq:invariant-distance}\\
 \mathcal W_2^{L^2}(\mu_h,\mu)&\leq Ch.
 \label{eq:invariant-W2}
\end{align}
The constants may depend on the fixed \(a,\lambda,\Phi\) and the Lyapunov
constants, but not on \(h\in(0,h_*]\).
\end{theorem}

\begin{proof}
\medskip\noindent\emph{Step 1: existence and uniform moments of
\(\mu_h\).}
Start the splitting approximation from \(0\in H^6_0\).
Proposition~\ref{prop:uniform-lyapunov} gives uniform \(H^6\) moment bounds.
Since \(H^6_0\) is compactly embedded in \(H^2_0\), the averaged laws
\[
 \frac1N\sum_{n=0}^{N-1}\delta_0Q_h^n
\]
are tight on \(H^2_0\).  The Feller property and the Krylov--Bogolyubov
argument yield an invariant probability measure \(\mu_h\).
Choose \(q\geq8d\bar q_6\).  The upper bound in
\eqref{eq:I6-coercive} gives
\(\bigl(I_6^+(x)\bigr)^{4d}\leq C(1+\|x\|_{H^6}^q)\).
Define on \(H^2_0\)
\begin{equation}
 \widehat{\mathcal V}(x)=
 \begin{cases}
 1+\|x\|_{H^6}^q,&x\in H^6_0,\\
 +\infty,&x\in H^2_0\setminus H^6_0.
 \end{cases}
 \label{eq:extended-invariant-weight}
\end{equation}
The functional \(\widehat{\mathcal V}\) is lower semicontinuous on \(H^2_0\).
Indeed, suppose that \(x_n\to x\) in \(H^2_0\) and
\(\liminf_n\widehat{\mathcal V}(x_n)<\infty\).
Choose a subsequence along which \(\widehat{\mathcal V}(x_n)\) converges
to this lower limit.  It is bounded in \(H^6_0\) and therefore has a
further subsequence converging weakly in \(H^6_0\).
The limit is \(x\), by convergence in \(H^2_0\).
Weak lower semicontinuity of the \(H^6\) norm then gives
\[
 \widehat{\mathcal V}(x)
 \leq\liminf_{n\to\infty}\widehat{\mathcal V}(x_n).
\]
Proposition~\ref{prop:uniform-lyapunov} bounds its integrals against the
averaged laws uniformly in \(N\) and \(h\).  The endpoint estimate
\eqref{eq:common-endpoint-exp}, used at exponent \(8a\leq\eta_*\), also
bounds their integrals of \(e^{8aG}\), which is continuous on \(H^2_0\).
The Portmanteau theorem, applied along the subsequence defining \(\mu_h\),
bounds both integrals uniformly in \(h\).
Since \(\widehat{\mathcal V}=+\infty\) outside \(H^6_0\), this gives
\(\mu_h(H^6_0)=1\) and the moment bounds in
\eqref{eq:invariant-mixed-bound}.

\medskip\noindent\emph{Step 2: uniqueness of the invariant measure.}
We first establish a weighted mean-square contraction estimate in the
\(L^2\) norm for splitting trajectories with initial states in \(H^2_0\).
For \(x,y\in H^6_0\), use the same stochastic
convolution in both splitting updates.  Lemmas~\ref{lem:moving-ou-weight} and
\ref{lem:one-sided-kdv-stability}, together with the parameter choice in
\eqref{eq:common-damping-threshold}, give
\[
 \begin{split}
 &\mathbb E\left[
 e^{2aG(\mathcal L^\omega_{h,0}y)}
 \|\mathcal L^\omega_{h,0}x-\mathcal L^\omega_{h,0}y\|_{L^2}^2\right]\\
 &\quad\leq
 e^{2aC_0\lambda h}e^{2ae^{-3\lambda h/7}G(y)}e^{-2\lambda h}
 \|\mathcal K_hx-\mathcal K_hy\|_{L^2}^2\\
 &\quad\leq e^{-\kappa_0h}e^{2aG(y)}\|x-y\|_{L^2}^2,
 \end{split}
\]
where \(\kappa_0>0\) is defined in
\eqref{eq:uniform-contraction-rate}.  Write \(U^{n,x}\) and \(U^{n,y}\)
for the two splitting trajectories.  Conditional iteration yields
\begin{equation}
 \mathbb E\left[e^{2aG(U^{n,y})}
                    \|U^{n,x}-U^{n,y}\|_{L^2}^2\right]
 \leq e^{-\kappa_0nh}e^{2aG(y)}\|x-y\|_{L^2}^2.
 \label{eq:numerical-weighted-synchronization}
\end{equation}
For each fixed \(n\), the map \(x\mapsto U^{n,x}\) is almost surely
continuous on \(H^2_0\), since it is a finite composition of continuous
KdV maps, linear damping, and translations by \(H^2_0\)-valued stochastic
convolutions.
By the density of \(H^6_0\) in \(H^2_0\), continuity of this map
and of \(G\), and Fatou's lemma,
\eqref{eq:numerical-weighted-synchronization} holds for all \(x,y\in H^2_0\).

If \(\varphi\) is bounded and Lipschitz with respect to the \(L^2\) norm,
write \([\varphi]_{\operatorname{Lip}(L^2)}\) for its Lipschitz constant.
Estimate~\eqref{eq:numerical-weighted-synchronization} and Cauchy--Schwarz
give, for every fixed \(x,y\in H^2_0\),
\[
 |Q_h^n\varphi(x)-Q_h^n\varphi(y)|
 \leq[\varphi]_{\operatorname{Lip}(L^2)}
 e^{-\kappa_0nh/2}e^{aG(y)}\|x-y\|_{L^2}\longrightarrow0.
\]
Let \(\rho_h\) be any invariant probability measure.  Invariance gives
\[
 \rho_h\varphi-\mu_h\varphi
 =\iint[Q_h^n\varphi(x)-Q_h^n\varphi(y)]
          \,\rho_h(\mathrm dx)\mu_h(\mathrm dy).
\]
The integrand converges to zero for each \(x,y\) and is bounded by
\(2\|\varphi\|_\infty\).  Dominated convergence gives
\(\rho_h\varphi=\mu_h\varphi\).
Since bounded Lipschitz functions determine probability measures on
\(L^2_0\), the identity \(\rho_h\varphi=\mu_h\varphi\) implies that
\(\rho_h\) and \(\mu_h\) coincide as measures on \(L^2_0\).
The Borel sigma-algebra of \(H^2_0\) is the restriction of that of
\(L^2_0\), so \(\rho_h=\mu_h\) on \(H^2_0\).

\medskip\noindent\emph{Step 3: the weighted dual error.}
Equations~\eqref{eq:exact-invariant-integrability} and
\eqref{eq:invariant-mixed-bound} give the required integrability for
\(\mu\) and \(\mu_h\), respectively.  Hence
\(\mathcal D_{d_a}(\mu_h,\mu)\) is finite.

Using \(d_a(x,0)\leq C_a e^{2aG(x)}\),
\eqref{eq:invariant-mixed-bound} and
\eqref{eq:weak-observable-integrability} also give
\[
 \int d_a(x,0)\,(\mu_hP_t)(\mathrm dx)<\infty,
 \qquad t\geq0.
\]
Since \(\mu P_t=\mu\), the dual-distance contraction
\eqref{eq:dual-distance-contraction} applies to \(\mu_h\) and \(\mu\).

Set
\begin{equation}
 N_h=\left\lceil\frac{4\log2}{\lambda h}\right\rceil,
 \qquad T_h=N_hh.
 \label{eq:integer-block}
\end{equation}
Then \(2e^{-\lambda T_h/2}\leq1/2\).
By \eqref{eq:invariant-mixed-bound}, the family \((\mu_h)\) satisfies the
moment condition \eqref{eq:weak-initial-order-one}.
Applying Theorem~\ref{thm:uniform-weak} with \(\nu_h=\mu_h\) and using the
definition of \(\mathcal D_{d_a}\), we obtain
\[
 \mathcal D_{d_a}(\mu_hQ_h^{N_h},\mu_hP_{T_h})\leq Ch.
\]
Using invariance and the triangle inequality, we combine this weak error
estimate with the contraction bound \eqref{eq:dual-distance-contraction}
to obtain
\[
 \begin{split}
 \mathcal D_{d_a}(\mu_h,\mu)
 &=\mathcal D_{d_a}(\mu_hQ_h^{N_h},\mu P_{T_h})\\
 &\leq\mathcal D_{d_a}(\mu_hQ_h^{N_h},\mu_hP_{T_h})
      +\mathcal D_{d_a}(\mu_hP_{T_h},\mu P_{T_h})\\
 &\leq Ch+\frac12\mathcal D_{d_a}(\mu_h,\mu).
 \end{split}
\]
Rearrangement proves \eqref{eq:invariant-distance}.

\medskip\noindent\emph{Step 4: the \(L^2\)-Wasserstein error.}
The measures \(\mu_h\) and \(\mu\) have finite second \(L^2\)-moments by
\eqref{eq:invariant-mixed-bound} and
\eqref{eq:exact-invariant-integrability}, respectively.
For each fixed \(h\), \eqref{eq:dual-distance-contraction} and invariance of
\(\mu\) give
\[
 \mathcal D_{d_a}(\mu_hP_{nh},\mu)
 \leq2e^{-\lambda nh/2}\mathcal D_{d_a}(\mu_h,\mu)
 \longrightarrow0.
\]
Since \(d_a(x,y)\geq\|x-y\|_{L^2}\), the integrals of every bounded
function that is Lipschitz with respect to the \(L^2\) norm converge to
their integrals under \(\mu\).  This suffices to conclude that the
probability measures \(\mu_hP_{nh}\) converge weakly to \(\mu\) on
\(L^2_0\) with its norm topology.
By \eqref{eq:invariant-mixed-bound}, Theorem~\ref{thm:uniform-weak} applies
with \(\nu_h=\mu_h\).  Invariance of \(\mu_h\) under \(Q_h\) and
\eqref{eq:uniform-transient-W2} give
\[
 \sup_{n\geq0}\mathcal W_2^{L^2}(\mu_h,\mu_hP_{nh})\leq Ch.
\]
Lower semicontinuity of \(\mathcal W_2^{L^2}\) with respect to weak
convergence then yields
\[
 \mathcal W_2^{L^2}(\mu_h,\mu)
 \leq\liminf_{n\to\infty}\mathcal W_2^{L^2}(\mu_h,\mu_hP_{nh})
 \leq Ch.
\]
This proves \eqref{eq:invariant-W2}.
\end{proof}
\section{Numerical illustrations}
\label{sec:computation}

We first examine the breakdown of the classical Burgers step, then test the
discrete \(L^2\)-moment identity, and finally compare fitted convergence orders at
short and long final times for two damping values.  These experiments illustrate
the corresponding results of the analysis.

The analytical results concern the time-discrete KdV--OU splitting with
exact substeps, in which space is left continuous.  The computations use a Fourier truncation,
finite-dimensional noise, and Monte Carlo sampling; the KdV--OU experiments
also use a Runge--Kutta integrator for the KdV flow.

We represent a real-valued function by
\[
 u_K(x)=\sum_{|k|\leq K}\widehat u_k e^{\mathrm i kx},
 \qquad \widehat u_{-k}=\overline{\widehat u_k}.
\]
For the KdV--OU experiments, \(N\) denotes the number of equispaced
collocation points and \(K\) the Fourier cutoff.  We evaluate the nonlinear
term using a Fourier pseudospectral method with the \(2/3\) dealiasing rule.
We remove the linear Airy evolution by an integrating factor and apply the
classical fourth-order Runge--Kutta method to the transformed Fourier system.
Every stochastic convolution is sampled exactly.  The deterministic initial datum
in both KdV--OU experiments is
\[
 u_0(x)=0.50\sin x+0.20\cos(2x)-0.08\sin(3x).
\]
The Burgers experiment in Section~\ref{sec:numerical-burgers} starts from
zero.  The Burgers steps are computed by the method of characteristics.

\subsection{Breakdown of the classical Airy--Burgers splitting}
\label{sec:numerical-burgers}

Theorem~\ref{thm:burgers-breakdown} shows that the classical Airy--Burgers
splitting almost surely breaks down after finitely many steps.  We examine
the first breakdown step for the Airy-first Lie--Trotter splitting at a fixed
step size.  We start from zero, fix \(h=0.125\), and use
\[
 \lambda=40,
 \qquad
 \Phi\,\mathrm dW=\gamma\cos(x)\,\mathrm d\beta_1+\gamma\sin(x)\,\mathrm d\beta_2,
 \qquad
 \gamma\approx21.982.
\]
For the present single-mode noise, the quantities in the proof of
Theorem~\ref{thm:burgers-breakdown} are
\[
 q_h=\frac{\gamma^2}{2\lambda}(1-e^{-2\lambda h}),
 \qquad
 p_h=\Psi\!\left(-\frac{1}{h\sqrt{q_h}}\right).
\]
With \(Z_h\) as defined in Section~\ref{sec:burgers-obstruction}, we choose
\(\gamma\) so that the Burgers step with initial datum \(Z_h\) breaks down
with probability \(0.005\).  For this single Fourier mode,
\(-\inf_{x\in\mathbb T}\partial_xZ_h(x)\) is the Euclidean norm of two
independent \(\mathcal N(0,q_h)\) coefficients, so
\[
 \mathbb P\!\left(\inf_{x\in\mathbb T}\partial_xZ_h(x)\leq-h^{-1}\right)
 =\exp\!\left\{-\frac{1}{2h^2q_h}\right\}=0.005.
\]
Here \(q_h\approx6.04\) is the pointwise variance of \(\partial_xZ_h\), and
\(p_h\) is the probability that this derivative is at most \(-h^{-1}\) at a
fixed spatial point.  Since \(\lambda h=5\), the contribution from the preceding state is
strongly damped by the factor \(e^{-5}\), while \(q_h\) is close to the
stationary variance \(\gamma^2/(2\lambda)\).

Let \(V_n\) denote the output of the stochastic Airy--damping part of step
\(n\), and define
\[
 N_{\mathrm{bd}}
 =\inf\left\{n\geq1:\inf_{x\in\mathbb T}\partial_xV_n(x)\leq-h^{-1}\right\}.
\]
Each path is simulated until its first breakdown or until
\(N_{\max}=2000\), whichever occurs first.  The computation uses
2000 independent paths, Fourier cutoff \(K=63\), and 512 grid points for
inversion of the characteristic map, followed by Fourier projection.
To estimate the minimum spatial derivative, we first
evaluate \(\partial_xV_n\) on 1024 equispaced grid points and then refine the
location of the minimum by a local Newton iteration.

The proof of Theorem~\ref{thm:burgers-breakdown} gives
\[
 \mathbb P(N_{\mathrm{bd}}\leq N)\geq1-(1-p_h)^N.
\]
Figure~\ref{fig:burgers-failure} compares the empirical cumulative
distribution function of \(N_{\mathrm{bd}}\) with this lower bound.  The
empirical median is 134.  For comparison, if the contribution from the
preceding state is neglected, the independent stochastic convolutions give
a geometric distribution for the first breakdown step, with parameter
\(0.005\) and median \(\lceil\log(1/2)/\log(0.995)\rceil=139\).
The proportions of paths that have broken down by
\(N=500\), \(1000\), and \(2000\) are \(0.924\), \(0.9945\), and \(1\),
respectively.  A separate comparison of the Fourier cutoffs \(K=63\) and
\(K=127\) uses 400 paths, \(N_{\max}=1000\), and the same noise realizations
for both cutoffs.  The same three paths survive all 1000 steps, and each of
the remaining 397 paths has the same first breakdown step at both cutoffs.

\begin{figure}[!htbp]
 \centering
 \includegraphics[width=0.72\textwidth]{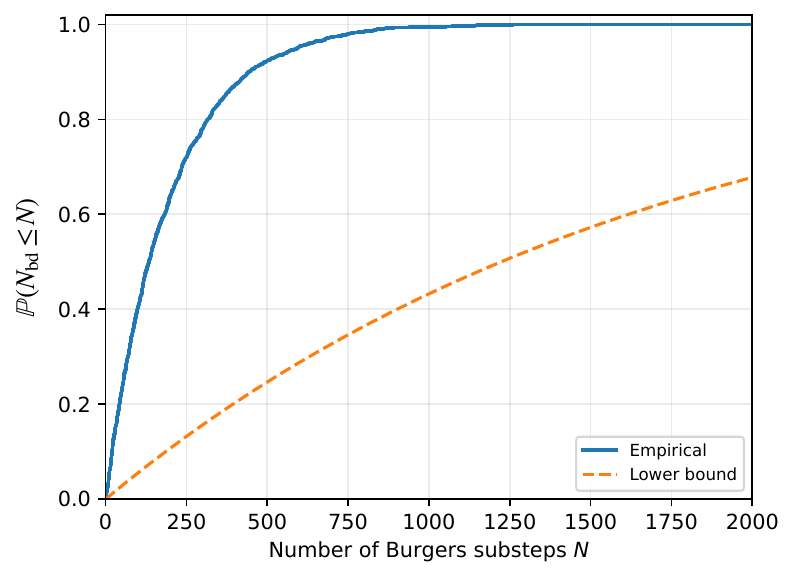}
 \caption{Empirical cumulative distribution function of \(N_{\mathrm{bd}}\)
 for \(h=0.125\), based on 2000 independent sample paths.  The dashed curve
 shows the theoretical lower bound \(1-(1-p_h)^N\).}
 \label{fig:burgers-failure}
\end{figure}

This experiment illustrates the accumulation of breakdown probability over
repeated steps at fixed \(h\).  For smaller noise amplitudes satisfying the
same derivative nondegeneracy condition, the theorem still gives eventual
breakdown, but the number of steps required may be much larger.  The recorded
event is the loss of classical solvability over the prescribed Burgers
step.  The stochastic KdV solution itself remains globally defined.

\FloatBarrier
\subsection{Discrete \texorpdfstring{\(L^2\)}{L2}-moment identity}

Proposition~\ref{prop:l2-balance} gives an exact identity for the second
moment of the splitting approximation.  We compare the second moment given
by this identity with sample means of the squared \(L^2\)-norm.

We use the finite-mode noise
\begin{equation}
 \Phi\,\mathrm dW
 =\sum_{k=1}^{8}\phi_k
 \{\mathrm dB_k e^{\mathrm i kx}+\mathrm d\overline{B_k}e^{-\mathrm i kx}\},
 \qquad
 \phi_k=0.18(1+k^2)^{-2},
 \label{eq:computational-noise}
\end{equation}
where the independent complex Brownian motions satisfy
\(\mathbb E[|\mathrm dB_k|^2]=\mathrm dt\).
For the normalization in \eqref{eq:computational-noise},
\(\|\Phi\|_{L_2(\mathfrak U,L^2)}^2=4\pi\sum_{k=1}^8\phi_k^2\).
The identity \eqref{eq:split-l2-iteration} holds for every \(\lambda>0\).
For the numerical illustration, we take \(\lambda=0.1\), \(T=25\), a fixed
splitting step size \(h=2^{-6}\), 96 collocation points with cutoff \(K=31\),
a maximum Runge--Kutta step size of \(2^{-8}\) for the KdV flow, and
\(M=200\) independent sample paths.

Figure~\ref{fig:l2-balance} compares the sample means at selected times
\(t_n\) with the curve defined by
\[
 m(t)=e^{-2\lambda t}\|u_0\|_{L^2}^2
 +\frac{1-e^{-2\lambda t}}{2\lambda}
  \|\Phi\|_{L_2(\mathfrak U,L^2)}^2.
\]

\begin{figure}[!htbp]
 \centering
 \includegraphics[width=0.72\textwidth]{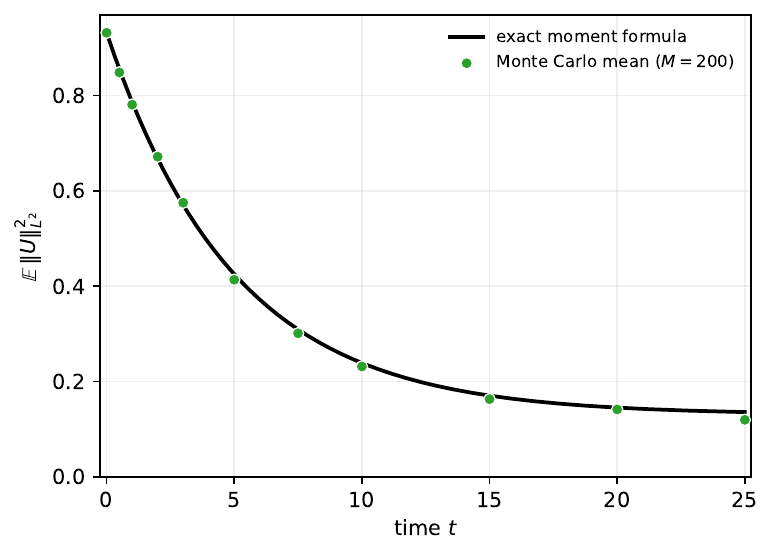}
 \caption{Discrete \(L^2\)-moment identity for \(\lambda=0.1\) and
 \(h=2^{-6}\).  The black curve is \(m(t)\), and the green markers show
 sample means of the squared \(L^2\)-norm over \(M=200\) independent paths
 up to \(T=25\).}
 \label{fig:l2-balance}
\end{figure}

At \(t=t_n\), \(m(t_n)\) equals the second moment in
\eqref{eq:split-l2-iteration}.  Here \(\|u_0\|_{L^2}^2\approx0.931168\), and
the limiting value of \(m(t)\) is
\(\|\Phi\|_{L_2(\mathfrak U,L^2)}^2/(2\lambda)\approx0.130726\).
The root-mean-square (RMS) difference between the sample means and \(m(t)\),
taken over the reported times, is about one percent of
\(m(0)-\lim_{t\to\infty}m(t)\).  The differences at positive times are of the
same order as the estimated Monte Carlo standard errors.

\FloatBarrier
\subsection{Observed convergence orders at short and long times}

The strong convergence results in Section~\ref{sec:strong-convergence}
require sufficiently large damping.  We compare \(\lambda=0.5\) and
\(\lambda=2\) to examine the effect of damping on the observed convergence
orders, keeping the initial datum, noise covariance, spatial discretization,
and step sizes fixed.  The noise is restricted to the Fourier pair
\(k=\pm1\), with
\[
 \Phi\,\mathrm dW=\phi_1
 \{\mathrm dB_1 e^{\mathrm i x}+\mathrm d\overline{B_1}e^{-\mathrm i x}\},
 \qquad \phi_1=5,
 \qquad \mathbb E[|\mathrm dB_1|^2]=\mathrm dt.
\]
We use \((N,K)=(128,42)\), the six splitting step sizes
\[
 h\in\{2^{-4},2^{-5},2^{-6},2^{-7},2^{-8},2^{-9}\},
\]
and a maximum Runge--Kutta step size of \(2^{-12}\) for the KdV flow.
The reference approximation is computed with \(h_{\mathrm{ref}}=2^{-13}\) on the
same spatial grid and with the same Runge--Kutta integrator.  For each damping
value, we use \(M=200\) independent sample paths and report results at selected
final times up to \(T=1500\).

For each damping value and Monte Carlo realization, the approximations at
all tested step sizes and the reference approximation are driven by the
same Wiener process.
Let \(\zeta_j^{\mathrm{ref}}\) denote the exact stochastic convolution over
\([jh_{\mathrm{ref}},(j+1)h_{\mathrm{ref}}]\), with \(j\geq0\).
For a tested step size \(h\), set \(N_{\mathrm c}=h/h_{\mathrm{ref}}\).
The convolution over the coarse interval \([jh,(j+1)h]\) is
\[
 \zeta^{(N_{\mathrm c})}_j=\sum_{i=0}^{N_{\mathrm c}-1}
 e^{-\lambda(N_{\mathrm c}-1-i)h_{\mathrm{ref}}}
 \zeta^{\mathrm{ref}}_{jN_{\mathrm c}+i}.
\]
Let \(U_h^{(m)}(T)\) denote the fully discrete approximation at time \(T\)
for the \(m\)-th sample path, \(m=1,\ldots,M\).
The RMS error relative to the reference approximation is
\[
 E_{\mathrm{ref}}(h,T)
 =\left(\frac1M\sum_{m=1}^M
 \|U_h^{(m)}(T)-U_{h_{\mathrm{ref}}}^{(m)}(T)\|_{L^2}^2\right)^{1/2}.
\]
For each damping value and final time \(T\), \(\hat p(T)\) is the
least-squares slope of
\(\log E_{\mathrm{ref}}(h,T)\) against \(\log h\), using all six step sizes.
The same error definition and fitting procedure are used in both
convergence figures.

Figure~\ref{fig:strong-convergence} shows the RMS errors \(E_{\mathrm{ref}}(h,T)\) for
\(\lambda=2\) at \(T=0.5\) and \(T=1000\).  The fitted orders are
\(1.004\) and \(0.987\), with coefficients of determination
\(R^2=0.9991\) and \(R^2=0.9996\),
respectively.  Both final times give fitted orders close to one.

\begin{figure}[!htbp]
 \centering
 \includegraphics[width=0.78\textwidth]{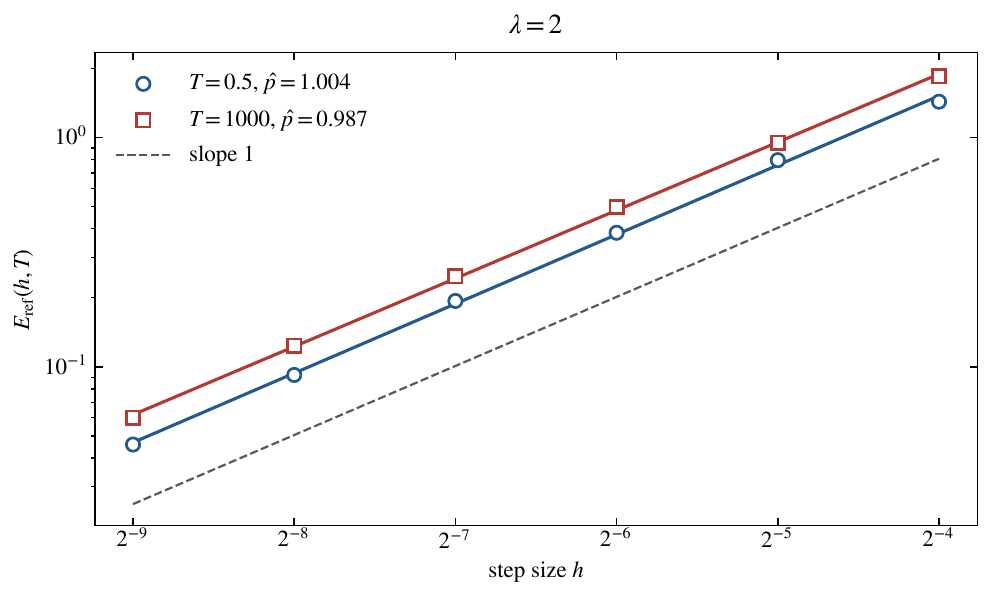}
 \caption{RMS errors \(E_{\mathrm{ref}}(h,T)\) for \(\lambda=2\) at
 \(T=0.5\) (blue circles) and \(T=1000\) (red squares), with
 \((N,K)=(128,42)\), \(M=200\), and \(h_{\mathrm{ref}}=2^{-13}\).
 The solid lines are least-squares fits over all six step sizes; the dashed
 line has slope one.}
 \label{fig:strong-convergence}
\end{figure}

Figure~\ref{fig:reference-order-vs-time} shows the fitted convergence orders
\(\hat p(T)\) at the selected final times for the two damping values.  For \(\lambda=2\), the
orders at the sampled times \(80\leq T\leq1500\) range from \(0.971\) to
\(1.008\).  For \(\lambda=0.5\), the fitted order decreases after the initial short
interval and then fluctuates.  Over the ten final times
\(T=1050,1100,\ldots,1500\), its mean is \(0.499\), with values ranging from
\(0.350\) to \(0.625\).  These are fits over a fixed range of step sizes
against a fixed reference approximation.  They do not determine the
asymptotic order as \(h\to0\).

\begin{figure}[!htbp]
 \centering
 \includegraphics[width=0.94\textwidth]{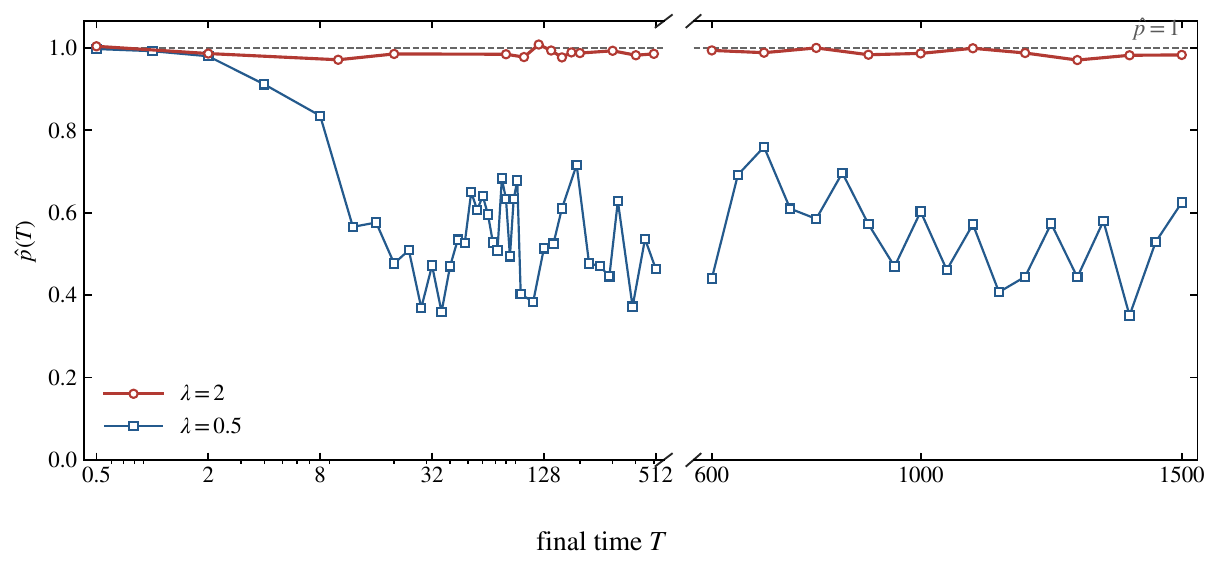}
 \caption{Fitted convergence orders \(\hat p(T)\) for \(\lambda=2\)
 (red circles) and \(\lambda=0.5\) (blue squares), up to \(T=1500\).
 Each order is fitted from the six step sizes \(h=2^{-4},\ldots,2^{-9}\),
 using \(M=200\) sample paths and \(h_{\mathrm{ref}}=2^{-13}\).
 The horizontal axis is logarithmic for \(0.5\leq T\leq512\) and linear
 for \(600\leq T\leq1500\).  The dashed line marks order one.}
 \label{fig:reference-order-vs-time}
\end{figure}

Since \(\lambda_{\mathrm{uni}}\geq1\), the value \(\lambda=0.5\) does not satisfy
the sufficient damping condition in Theorem~\ref{thm:uniform-strong}.
The constants defining \(\lambda_{\mathrm{uni}}\) and \(h_*\) in
\eqref{eq:common-damping-threshold}--\eqref{eq:common-step-threshold} have not
been evaluated numerically.  For \(\lambda=2\), the fitted orders remain
close to one over the reported interval, consistent with the uniform-in-time
first-order estimate.  This comparison illustrates the role of damping
without locating the theoretical thresholds.  The smaller fitted orders
for \(\lambda=0.5\) do not establish failure of convergence.

\FloatBarrier
\section{Conclusion}

We have studied the long-time approximation of the periodic damped stochastic
KdV equation by a KdV--Ornstein--Uhlenbeck Lie--Trotter splitting. For
sufficiently large damping, the splitting approximation converges with strong
and weak order \(1\) uniformly in time, while strong order \(1\) on finite time
intervals holds under a weaker damping condition. Under stronger assumptions,
the splitting approximation admits a unique invariant probability measure on
\(H^2_0\), which approximates the invariant measure of the exact dynamics with
order \(1\) in both a weighted dual distance and the \(2\)-Wasserstein distance
induced by the \(L^2\) norm. As a complementary result, we have shown that,
under a nondegeneracy condition on the noise, the classical Airy--Burgers
splitting almost surely breaks down after finitely many steps. The numerical
illustrations are consistent with the corresponding theoretical results over
the simulated long-time intervals. The present estimates provide an avenue for
long-time simulation of the damped sKdV equation when suitable numerical
realizations are used for the splitting substeps. Future work will extend
the analysis to fully discrete approximations.

\backmatter
\section*{Statements and Declarations}

\paragraph{Funding.}
This work was partially supported by the National Natural Science
Foundation of China (Grant Nos.~11871271, 12001283 and 12501546) and Hunan Provincial Natural Science Foundation of China (Grant No. 2026JJ60332).

\paragraph{Competing Interests.}
The authors have no competing interests to declare.

\paragraph{Data Availability.}
The data supporting the numerical experiments are available from the
corresponding author upon reasonable request.

\bibliography{references_polished}

\end{document}